\documentclass{article}
\usepackage{graphicx,subcaption}
\usepackage{mathtools}
\mathtoolsset{showonlyrefs}
\usepackage{amssymb}
\usepackage{amsmath}
\usepackage{amsthm}
\usepackage{xcolor,soul}
\usepackage{fancyhdr}
\usepackage{epstopdf}
\usepackage[colorlinks=true,citecolor=magenta,linkcolor=blue]{hyperref}
\usepackage{enumerate}

\makeatletter
\let\@fnsymbol\@arabic
\makeatother

\usepackage[]{geometry}
\definecolor{greenEE}{RGB}{0, 128, 0}

\newtheorem{theorem}{Theorem}
\newtheorem{lemma}{Lemma}
\newtheorem{proposition}{Proposition}
\newtheorem{remark}{Remark}

\newcommand{\dd}{\mathrm{d}}

\usepackage{tikz}

\usetikzlibrary{positioning}
\usetikzlibrary{arrows}
\usetikzlibrary{arrows.meta}
\usetikzlibrary{calc}
\usetikzlibrary{shapes}

\graphicspath{ {./figures/} }

\title{On the optimal control of entry-exit phenomena\\ in planar fast-slow dynamical systems}
\author{Jacopo Borsotti\thanks{School of Resource and Environmental Management, Simon Fraser University, 8888 University Drive, V5A 1S6 Burnaby, Canada, {\tt jacopo\_borsotti@sfu.ca}}, Christian Kuehn\thanks{Technical University of Munich, School of Information Computation and Technology, Department of Mathematics, Boltzmannstr. 3, 85748 Garching, Germany, {\tt ckuehn@ma.tum.de}}, and Mattia Sensi\thanks{Dipartimento di Matematica, Università degli Studi di Trento, Via Sommarive 14, 38123 Povo (Trento), Italy, {\tt mattia.sensi@unitn.it}}}
\date{\today}

\begin{document}

\maketitle

\begin{abstract}
    We study a geometric singular perturbation theory (GSPT) approach to optimal control problems involving entry-exit phenomena in planar fast-slow dynamical systems. In contrast to previous work, we assume that the control acts only on the slow dynamics and we consider minimum-time optimal control problems (i.e., we aim at minimizing the time to reach a certain target). We apply this method to a simplified fast-slow toy model, characterized by a repetition of entry-exit processes, enabling a transparent analytical exploration of the control dynamics and allowing us to derive the key results explicitly. 
    
    First, we analyze each non-autonomous entry-exit phenomenon separately, and we derive the optimal way to control them. Second, by artificially generating a stable limit cycle, we determine under which conditions it is possible to reach the target. Finally, under such conditions, we combine all entry-exit processes together and we demonstrate the existence of an optimal control. Our approach mainly exploits techniques of GSPT and dynamic programming. Indeed, we show that our problem can be formulated in a way resembling the Bellman equation, even though not in its standard setting. Numerical simulations, built on such Bellman-like formulation, illustrate our analytical results. In particular, the optimal control is derived with an algorithm that converges in a finite number of steps. Lastly, we extend this novel approach to more complex planar fast-slow dynamical systems. 
\end{abstract}

\section{Introduction}

Systems of ordinary differential equations (ODEs) are a classical tool for modeling a wide range of real-world phenomena. The observation that many real-world systems evolve on different timescales has led to the development of fast-slow dynamical systems, characterized by a small singular parameter $0<\varepsilon\ll 1$ representing the separation between timescales. Since, in some settings, $\varepsilon$ can be viewed as a singular perturbation parameter, geometric singular perturbation theory (GSPT) provides a natural framework for analyzing fast-slow systems \cite{40}. GSPT has been coined within the seminal work of Neil Fenichel \cite{23} and has since been complemented by several more advanced techniques, necessary to treat less regular situations, like the \emph{entry-exit function} \cite{35}.

A general approach to the analysis of fast-slow systems is to distinguish between regions of the phase space where $\varepsilon$ can be neglected, corresponding to the fast timescale, and regions where its effects become relevant, which are in turn visible in the slow timescale. Orbits may switch repeatedly between these regions, giving rise to so-called \emph{entry-exit processes}. These phenomena have been studied extensively in a variety of settings, like epidemic \cite{jardon2021geometric,jardon2021geometric2,della2024geometric,kaklamanos2024geometric,bulai2024geometric}, neuroscience \cite{rodrigues2016time,sensi2023slow}, and predator-prey \cite{allee} models.

Naturally, growing interest has also been devoted to the introduction of control strategies for fast-slow dynamical systems. The general idea is to exploit the geometric structure of these systems to develop \emph{ad hoc} control techniques. Similarly, several works have investigated non-autonomous fast-slow systems, which are particularly interesting in view of the multiple timescales inherent to these models. Examples in biological modeling include \cite{li2018bifurcation,beron2021nonlinear,o2023rate,alraddadi2024analysis,jelbart2024rate,longo2024transition,alraddadi2025asymmetric,longo2025tracking,nunez2026averaging}.

In this work, our aim is to develop optimal control strategies for entry-exit phenomena. In particular, we consider the combination of a generic number of entry-exit processes. The goal of the optimal control problem is to reach a prescribed target in phase space in the minimum possible (slow) time. We achieve this by designing suitable controls for each entry-exit process individually. The motivation for this problem comes from the classical ODE setting, where optimal control strategies are commonly developed to steer a system toward a prescribed target; see, for example, \cite{groppi1} for an application to an epidemiological model.

Our main result is the following theorem (in the remainder of this work, $\dot{}$ and $'$ denote differentiation with respect to the time variables $t$ and $\tau=\varepsilon t$, respectively, which we omit from the phase variables for ease of notation).

\begin{theorem}\label{teo:general}
    Consider the following planar model, evolving inside some set $\Delta \subset \mathbb{R} \times [0, +\infty)$: 
    \begin{align} \label{eq:model_generic_teo1}
    \begin{aligned}
        \dot{y} &= \varepsilon u(t) g(y,z,\varepsilon) + z h(y,z,\varepsilon), \\ 
        \dot{z} &= z f(y,z,\varepsilon), 
    \end{aligned}
    \end{align}
    where: 
    \begin{enumerate} [(a)]
        \item $0<\varepsilon\ll 1$ and $u(t)$ represents an external control satisfying $u(t) \in [u_m, u_M]$ with $0<u_m<u_M$, $u_m, u_M \in \mathcal{O}(1)$, and $u_M-u_m \in \mathcal{O}(1)$; 
        \item $f,g,h \colon \Omega \subset  \mathbb{R}^3 \to \mathbb{R}$ are differentiable, $\textnormal{sign}(f(y,0,0))=\textnormal{sign}(y)$, and $g(y,0,0)>0$; 
        \item the critical manifold \cite[Definition 3.2]{wechselberger2020geometric} of \eqref{eq:model_generic_teo1} is $\{z=0\}$ and the fast subsystem of \eqref{eq:model_generic_teo1} \cite[Definition 3.3]{wechselberger2020geometric} generates a map 
        \begin{equation} \label{eq:P_fast_teo1}
            \Pi_{\textnormal{fast}} \colon \{y \in (0, +\infty) \cap \pi_y(\Delta)\} \to \{ y \in (-\infty, 0) \cap \pi_y(\Delta)\}, 
        \end{equation}
        i.e., the orbits of the fast flow are heteroclinic to two points (of opposite sign) on $\pi_y(\Delta)$, which is the projection of $\Delta$ along the $y$-axis. 
    \end{enumerate}
    Then,  
    \begin{enumerate}[(i)]
        \item System \eqref{eq:model_generic_teo1} shows a repetition of fast flows, described by the map $\Pi_\textnormal{fast}$ \eqref{eq:P_fast_teo1}, and slow dynamics, described by entry-exit phenomena occurring $\mathcal{O}(\varepsilon^2)$-close the critical manifold $\{z=0\}$. Moreover, the fast flows are independent on $u$, allowing us to consider it only during the slow dynamics through $U(\tau) \coloneqq u(t)$. 
        \item For each entry point in the slow flow $y_\infty<0$, an $\mathcal{O}(\varepsilon)$-approximation of the (slow) exit time $\tau_E$ is given by the non-trivial solution of the integral equation 
        \begin{equation} 
            \int_0^{\tau_E} f(y(s), 0, 0) \dd s = 0, 
        \end{equation}
        where $y(s)$ evolves according to $y'=U(\tau) g(y,0,0)$. Similarly, an $\mathcal{O}(\varepsilon)$-approximation of the exit point $y_E>0$ is given by the non-trivial solution of the integral equation 
        \begin{equation} 
            \int_{y_\infty}^{y_E} \frac{f(y, 0, 0)}{U(y)g(y,0,0)} \dd y = 0, 
        \end{equation}
        where we expressed $U$ as a function of $y$ because the latter is strictly increasing during this process, and hence invertible. Moreover, in the limit $\varepsilon \to 0$, the maximum exit point $y_E^\textnormal{max}$ is obtained by setting 
        \begin{align} \label{eq:maximum_teo1}
            U(y)=
            \begin{cases}
                u_m \quad & \textnormal{if } y < 0, \\
                u_M \quad & \textnormal{if } y > 0, \\
            \end{cases}
        \end{align} 
        hence we can define the map
        \begin{equation} 
            \Pi_{\textnormal{slow}}^{\max} \colon \{y \in (-\infty, 0) \cap \pi_y(\Delta)\} \to \left\{ y \in \left(0, +\infty\right) \cap \pi_y (\Delta) \right\}
        \end{equation}
        which maps any given $y_\infty$ to the corresponding $y_E^{\max}$. The expression of $U(\tau)$ that leads to the minimum possible exit point $y_E^{\min}$ is obtained by reversing $u_m$ and $u_M$ in \eqref{eq:maximum_teo1}. 
        \item For any exit point $y_E \in [y_E^{\min}, y_E^{\max}]$, in the limit $\varepsilon \to 0$, there exists an optimal choice of $U$ that leads to it in the minimum possible exit time. In particular, there exists a non-negligible set $\Phi \subset [y_\infty, y_E]$ such that all optimal choices of $U$ coincide almost everywhere on $\Phi$ and attain values in $\{u_m, u_M\}$. On the other hand, on $\Phi^\complement=[y_\infty, y_E] \setminus \Phi$, $U$ could attain different values in $[u_m, u_M]$, potentially leading to the non-uniqueness of the optimal control if $\Phi^\complement$ has positive measure.
        \item Consider a target value for the $z$ variable $0<\hat{z}\in\mathcal{O}(1)$ and an initial condition $(y_0,z_0)$, $z_0>0$. In the limit $\varepsilon \to 0$, there exist suitable choices of $U(\tau)$ that allow the target to be reached if and only if it can be reached by the strategy consisting of the combination of the maps $\Pi_{\textnormal{fast}} \circ \Pi_{\textnormal{slow}}^{\max}$ (note that such map is initially applied to the first entry point $y_\infty$ related to the initial fast flow, which is independent on $U$, starting from the aforementioned initial condition $(y_0, z_0)$). If $\hat{z}$ can be reached starting from $y_\infty$, then this strategy leads to it in the minimum number of entry-exit phenomena $1 \le N(y_\infty) < +\infty$. 
        \item Assume that it is possible to reach $\hat{z}$ starting from $y_\infty$; then, in the limit $\varepsilon \to 0$, the following hold: 
        \begin{enumerate} [1.]
            \item For any $M \ge N(y_\infty)$,  among all possible strategies that lead to $\hat z$ in at most $M$ entry-exit processes, there exists an optimal one that minimizes the total slow time. 
            \item Assume that there exists $\tau_{\min}=\tau_{\min}(y_\infty)>0$ such that any strategy leading to $\hat{z}$ and containing an entry-exit process of duration less than $\tau_{\min}$ cannot achieve a total slow time arbitrarily close to the infimum over all strategies leading to the target. Then, there exists an optimal strategy that minimizes the slow time required to reach $\hat{z}$.
        \end{enumerate}
        Specifically, each entry point and exit point is chosen according to the mentioned strategies, subject to the constraints imposed by all possible dynamics of system \eqref{eq:model_generic_teo1}. On the other hand, fixing the entry and exit points, each entry-exit process is controlled by a control provided in (iii).
        \item The aforementioned optimal controls provide $\mathcal{O}(\varepsilon |\log \varepsilon|)$-approximations of the strategies that would correspond to a fixed $\varepsilon \ll 1$, provided that one is satisfied with reaching a target $\mathcal{O}(\varepsilon |\log \varepsilon|)$-close to $\hat{z}$. 
    \end{enumerate}
\end{theorem}

The statement of the main theorem is somewhat technical and its implications may therefore not be immediately apparent. For this reason, we first illustrate the result through a toy model characterized by dynamics that are common in biological models. In this setting, all relevant quantities can be derived explicitly, providing a transparent interpretation of the main result. Only later do we return to the general setting and prove Theorem \ref{teo:general}, referring back to the more digestible toy model which we unravel in full detail.

\subsection{The toy model} 

We consider the following non-autonomous fast-slow toy model in non-standard GSPT form \cite{wechselberger2020geometric}: 
\begin{align} \label{eq:model}
\begin{aligned}
    \dot y & = \varepsilon u(t) - z(y+b), \\
    \dot z & = zy, 
\end{aligned}
\end{align}
with $0<b \in \mathcal{O}(1)$. Moreover, as in Theorem \ref{teo:general}, $u(t) \in [u_m, u_M]$ with $0<u_m<u_M$, $u_m, u_M \in \mathcal{O}(1)$, and $u_M-u_m \in \mathcal{O}(1)$. System \eqref{eq:model} evolves in the region 
\begin{equation} \label{eq:Delta}
    \Delta \coloneqq \{(y,z) \in \mathbb{R}^2 \colon \; y \ge -b, \; z \ge 0\}.
\end{equation}
Indeed, one can easily observe that $\dot y|_{y=-b} = \varepsilon u(t)>0$ and $\dot z|_{z=0}=0$. Despite its simplicity, we will show that system \eqref{eq:model} presents interesting dynamics. Moreover, we remark that, if $u(t) \equiv \Bar{u} \in [u_m, u_M]$, then its dynamics resemble the ones of some SIRS \cite{jardon2021geometric} and predator-prey \cite{allee} models (up to a linear change of variables). 

\subsection{Related works}

A valuable theoretical starting point for the study of non-autonomous systems like \eqref{eq:model} can be found in \cite{berglund2000dynamic}, where it was realized that geometric methods from dynamics are very helpful for fast-slow problems with controls. We adopt this geometric viewpoint but our control theory analysis goes considerably beyond \cite{berglund2000dynamic}. Furthermore, in \cite{jardon2022controlling} the authors consider a planar system with a folded critical manifold, and provide analytical results on the possibility of locally stabilizing a canard cycle. As we will see, our approach results in a similar creation and stabilization of a canard cycle, even though the critical manifold in system \eqref{eq:model} is a line (namely $\{ z=0\}$) with the non-hyperbolic point $(y,z)=(0,0)$ separating its stable branch from its unstable branch. Finally, we mention on the control side there are very detailed asymptotic analysis approaches to fast-slow systems, particularly the works of Kokotovic et al. \cite{kokotovic1984applications,kokotovic1999singular,sepulchre2012constructive}. In this stream of work it was shown that fast-slow decomposition can often be applied to design control and study stability, particularly regarding Lyapunov functions, composite controllers with fast and slow pieces, order reduction, and optimal control. From the fast-slow systems perspective, the main tools we use are Fenichel theory and the entry-exit function (or slow divergence integral), which has attracted considerable attention in GSPT, see e.g. \cite{de2005time,dumortier2011slow,hsu2021relaxation,kaklamanos2025entry,huzak2025ergodicity,neishtadt2026maximal,christensen2026slow}. Our contribution basically leverages inside from all these previous works, combines them with the Bellman equation framework from control theory, and then treats completely an important class of fast-slow control problems arising in applications.

\subsection{Outline of the paper} 

This paper is organized as follows. We start in Section \ref{sec:auton} by analyzing the autonomous version of the toy model \eqref{eq:model}, obtained by setting $u(t) \equiv \Bar{u}$ for some $\Bar{u} \in [u_m, u_M]$. This preliminary analysis exploits standard techniques of GSPT, such as the entry-exit function, and allows us to rigorously characterize the behavior of the orbits. Specifically, we carry out a geometric proof that they converge toward a locally asymptotically stable focus after a finite number of entry-exit processes. 

These results are generalized in Section \ref{sec:non-auton}, where we consider $u$ a time-dependent quantity. Specifically, we discuss how to apply the entry-exit function in our non-autonomous setting and we derive an explicit optimal control strategy for each entry-exit process considered individually. Given an entry point, such strategy explains which exit points it is possible to reach and, chosen one of them, how to do so in the minimum possible time. Then, we apply the aforementioned optimal control to maximize the amplitude of all entry-exit processes, giving birth to a stable limit cycle, which prevents the convergence toward an equilibrium. 

Section \ref{sec:opt_cont} contains the main results of this work. Given a target $0 <\hat{z} \in \mathcal{O}(1)$, we use the previous results to study the problem of reaching it in the minimum possible time. First, we discuss which targets are possible to reach. This analysis mainly relies on the aforementioned ``artificial'' stable limit cycle of maximum possible amplitude. Then, given an admissible target $\hat{z}$, we formulate the optimal control problem with an approach resembling a dynamic programming problem, allowing for an easier mathematical tractability. We prove the existence of an optimal control leading to $\hat{z}$ and, in one case, its explicit expression. By reformulating our problem in a way resembling the Bellman equation \cite{bellman1954theory}, we construct an algorithm to numerically derive such optimal control. We conclude the section with two numerical examples. 

Finally, in Section \ref{sec:other_planar} we  discuss how to extend our approach to more complex planar fast-slow models. First, we explain which characteristics such models must have. Second, we show that the extension of our techniques is almost straightforward. The main difference is that one might not be able to derive explicitly the optimal control of each entry-exit process taken individually. However, it is still possible to prove an existence result, which is sufficient for our scope. In conclusion, we complete the proof of Theorem \ref{teo:general}. 

\section{Autonomous case} \label{sec:auton}

\subsection{Equilibria and nullclines}

Assume for now that $u(t) \equiv \Bar{u} \in [u_m, u_M]$. System \eqref{eq:model} becomes 
\begin{align} \label{eq:model_auton}
\begin{aligned}
    \dot y & = \varepsilon \Bar{u} - z(y+b), \\
    \dot z & = zy. 
\end{aligned}
\end{align}
System \eqref{eq:model_auton} is said to be in \emph{non-standard form}; we refer to  \cite{wechselberger2020geometric} for a monograph on the subject. System \eqref{eq:model_auton} always admits a unique equilibrium point in $\Delta$, namely
\begin{equation} \label{eq:equilib}
    \mathbf{e}=(e_y, e_z) = \left(0, \varepsilon\frac{\Bar{u}}{b}\right) . 
\end{equation}
Note that $e_z \in \mathcal{O}(\varepsilon)$. Since $\varepsilon \ll 1$, a simple calculation shows that $\mathbf{e}$ is a locally asymptotically stable focus, regardless of the values attained by $b$ and $\Bar{u}$. Indeed, the Jacobian of system \eqref{eq:model_auton} is given, in a generic point and evaluated in \eqref{eq:equilib} respectively, by
\begin{equation}
\mathcal{J}= \begin{pmatrix}
    -z & -y-b\\
    z &y
\end{pmatrix}, \qquad \mathcal{J}|_{\mathbf{e}}= \begin{pmatrix}
    -e_z & -b\\
    e_z &0
\end{pmatrix},
\end{equation}
the latter having negative trace and positive determinant, both in $\mathcal{O}(\varepsilon)$.

The equilibrium $\mathbf{e}$ \eqref{eq:equilib} arises from the intersection of the $z$-nullcline, composed of both axes, and the $y$-nullcline 
\begin{equation} \label{eq:y-nullcline}
    z = \varepsilon \frac{\Bar{u}}{y+b}\,. 
\end{equation}
Note that the latter is $\mathcal{O}(\varepsilon)$-close to the $y$-axis as long as $y\in-b+\mathcal{O}(1)$. 

\subsection{Multi-timescale analysis} \label{sec:multi_auton}

In this section, we want to exploit GSPT techniques to completely characterize the behavior of the orbits. We start by analyzing the fast dynamics of the models, then we will move to the slow dynamics, and finally we will combine them to completely understand the evolution of the orbits.  

\subsubsection{Fast formulation} \label{sec:fast_auton}

Setting $\varepsilon=0$ in \eqref{eq:model_auton} we obtain the corresponding fast subsystem: 
\begin{align} \label{eq:fast_model_auton}
\begin{aligned}
    \dot y & = - z(y+b), \\
    \dot z & = zy. 
\end{aligned}
\end{align}
The critical manifold is defined as the set of the equilibria of \eqref{eq:fast_model_auton} (see, e.g., \cite{wechselberger2020geometric}), namely 
\begin{equation} \label{eq:critical_man}
    \mathcal{C}_0 \coloneqq \{(y,z) \in \mathbb{R}^2 \;\colon \; z=0\}. 
\end{equation}
Denote with $(y_0,z_0)$, $z_0>0$ the initial conditions and define $y_\infty$ and $z_\infty$ as 
\begin{equation}
    y_\infty=\lim_{t \to +\infty} y(t) \quad \textnormal{and} \quad z_\infty=\lim_{t \to +\infty} z(t), 
\end{equation}
when these limit exist under the flow of system \eqref{eq:fast_model_auton}. 

\begin{proposition} \label{prop:conv_to_C0}
The trajectories of system \eqref{eq:fast_model_auton} converge to $\mathcal{C}_0$ \eqref{eq:critical_man} as $t \to +\infty$. In particular, $y_\infty \in [-b, y_0]$. 
\end{proposition}
\begin{proof}
Note that the solutions of the fast system \eqref{eq:fast_model_auton} naturally evolve inside $\Delta$ \eqref{eq:Delta}. Since $\dot{y} \le 0$, there exists $y_\infty \in [-b, y_0]$. Moreover, $\dot{y}+\dot{z} \le 0$, therefore there exist also $z_\infty \ge 0$. Since 
\begin{equation}
\begin{split}
    - \infty  < y_\infty + z_\infty - y_0 - z_0 = \int_0^{+\infty} \big(\dot{y}(t)+\dot{z}(t)\big)  \dd t  = -b \int_0^{+\infty} z(t)   \dd t  < 0, 
\end{split}
\end{equation}
then necessarily $z_\infty=0$. 
\end{proof} 

The following proposition provides an implicit expression for $y_\infty$. 

\begin{proposition} \label{prop:Gamma}
The quantity 
\begin{equation} \label{eq:Gamma}
    \Gamma(z,y) = b \log (y+b) - y - z
\end{equation}
is a constant of motion for system \eqref{eq:fast_model_auton}. Moreover, $y_\infty \in (-b, 0)$. 
\end{proposition}
\begin{proof}
By direct derivation with respect to time, one can see that $\dot \Gamma \equiv 0$, 
therefore 
\begin{equation} \label{eq:u_infty}
    b \log (y_\infty + b) - y_\infty = b \log (y_0 + b) - y_0 - z_0. 
\end{equation}
Define over $(-b, y_0)$ the function $h(x)= x-y_0 - z_0 - b \log \frac{x + b}{y_0 + b}$ and note that $h(y_0)<0$ while $\displaystyle \lim_{x \to -b} h(x) = +\infty$. Since $\frac{\dd h(x)}{\dd x} = 1-\frac{b}{x+b}$, $h$ decreases in $(-b, \min\{0, y_0\})$ and increases in $(\min\{0, y_0\}, y_0)$, hence there exists a unique zero of $h$, which by definition coincides $y_\infty$, in the open interval $(-b, 0)$. 
\end{proof} 

The constant of motion \eqref{eq:Gamma} is a modified version of a classical constant of motion for SIR models, see e.g. \cite{jardon2021geometric,allee}. Assume $z_0\in \mathcal{O}(\varepsilon)$, we can thus ignore it in \eqref{eq:u_infty} and derive the following relation between $y_\infty$ and $y_0$:
\begin{equation}\label{eq:riscrit}
    b \log (y_\infty + b) - y_\infty = b \log (y_0 + b) - y_0. 
\end{equation}
For future use, let us define the map 
\begin{equation} \label{eq:map_Pi_fast}
    \Pi_{\textnormal{fast}} \colon \{y \in (0, + \infty)\}  \to \{y \in (-b, 0)\}    
\end{equation}
that maps $y_0$ into $y_\infty$ according to \eqref{eq:riscrit}. 

\subsubsection{Slow formulation} 

Assume that an orbit of \eqref{eq:model_auton} reached an $\mathcal{O}(\varepsilon)$-neighborhood of the critical manifold $\mathcal{C}_0$, meaning $z \in \mathcal{O}(\varepsilon)$. We rescale $z$ as $z=\varepsilon x$ and apply a rescaling to the time variable, bringing the system to the slow timescale $\tau=\varepsilon t$ \cite{40}: 
\begin{align} \label{eq:slow_model_auton}
\begin{aligned}
    y' &= \Bar{u} - x(y+b), \\ 
    \varepsilon x' &= x y, 
\end{aligned}
\end{align}
where the $'$ now indicates the derivative with respect to the slow time $\tau$. 

If we look at system \eqref{eq:slow_model_auton} on the critical manifold $\mathcal{C}_0$, now determined by $x=0$, we obtain 
\begin{equation} \label{eq:slow_flow_auton}
    y'=\Bar{u}>0.  
\end{equation}
Hence, on $\mathcal{C}_0$, $y$ increases linearly.  

System \eqref{eq:slow_model_auton} is in standard GSPT form \cite{40}. The stability of $\mathcal{C}_0$ depends on the eigenvalue $\lambda=y$ associated to the fast variable $x$ \cite[Definition 3.1.3]{40}, i.e., the critical manifold is attractive for $y<0$, while it is repelling for $y>0$. As long as $x \in \mathcal{O}(\varepsilon)$ (i.e., $z \in \mathcal{O}(\varepsilon^2)$), in the limit $\varepsilon \to 0$, $y$ \eqref{eq:slow_model_auton} behaves like the solution to \eqref{eq:slow_flow_auton} and the stability of $\mathcal{C}_0$ determines whether $x$ tends to increase or decrease (see Tikhonov's Theorem \cite[Theorem 4.1.2]{2}). 

\begin{figure}[h!]
    \centering 
     \begin{tikzpicture}
 \node at (0,0) {\includegraphics[width=0.4\linewidth]{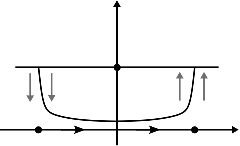}};
\node at (-0.3,2) {$x$};
\node at (-2.5,0.5) {$x=\Tilde{x}$};
\node at (-2.3,-1.9) {$y_\infty$};
\node at (2.2,-1.9) {$y_E$};
\node at (3.4,-1.9) {$y$};
  \end{tikzpicture}
    \caption{Visualization of the entry-exit map on the line $\{x=\Tilde{x}\}$, $\Tilde{x} \in \mathcal{O}(\varepsilon)$ illustrating delayed loss of stability. The point $(y,x)=(0,0)$ is the unique non-hyperbolic point on the critical manifold $\{x=0\}$. \label{fig:entry-exit}}
\end{figure}
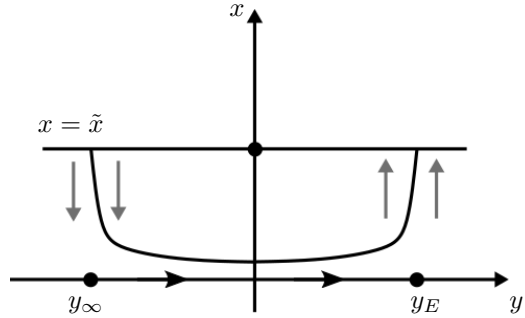

In order to understand the precise behavior of the orbits near the non-hyperbolic point $y=0$ of $\mathcal{C}_0$, we can exploit the entry-exit function \cite{35,hsu2017bifurcation,ai2020entryexit} which gives, in the form of a Poincaré map, an estimate of said behavior. The intuitive explanation is to ``balance'' the contraction accumulated during the flow close to the attractive part of the critical manifold with the repulsion accumulated during the flow close to its repelling part, a phenomenon commonly referred to as \emph{delayed loss of stability} \cite{kaklamanos2024geometric,SIRS}. Assume that an orbit of the perturbed system \eqref{eq:model_auton} enters a $\mathcal{O}(\varepsilon^2)$-neighborhood of $\mathcal{C}_0$ for some $y_\infty \in (-b, 0)$ (recall Proposition \ref{prop:Gamma}), the entry-exit function provides $\mathcal{O}(\varepsilon)$-approximations of the exit point $y_E \in (0, + \infty)$ and of the (slow) exit time $\tau_E$ (see Figure \ref{fig:entry-exit}). In particular, $y_E$ is the non-trivial solution to 
\begin{equation} \label{eq:exit_point_auton}
    \int_{y_\infty}^{y_E} \frac{y}{\Bar{u}} \dd y = 0,    \quad \textnormal{i.e.,} \quad y_E=-y_\infty, 
\end{equation}
while the (slow) exit time $\tau_E$ needed by $y$ to increase from $y_\infty$ to $y_E$ is the non-trivial solution to 
\begin{equation} \label{eq:exit_time_auton}
    \int_0^{\tau_E} y(\tau)  \dd\tau = 0, \quad \textnormal{i.e.,} \quad \tau_E=-\frac{2}{\Bar{u}}y_\infty,  
\end{equation}
where we have used that (recall \eqref{eq:slow_flow_auton}) 
\begin{equation}
    y(\tau)=y_\infty + \Bar{u}\tau.
\end{equation} 
Note, for future use, that the argument of the integral in \eqref{eq:exit_time_auton} is the eigenvalue $\lambda$ describing the stability of the critical manifold. Recall \eqref{eq:map_Pi_fast}; in the limit $\varepsilon \to 0$, we define the map 
\begin{equation} \label{eq:map_Pi_slow}
    \Pi_{\textnormal{slow}} \colon \{y \in (-b, 0)\} \to \{ y \in (0, b) \}
\end{equation}
that maps $y_\infty$ into $y_E=-y_\infty$. 

\begin{remark}
    The equilibrium $\mathbf{e}$ \eqref{eq:equilib} does not influence the slow dynamics. Indeed, it is $\mathcal{O}(\varepsilon)$-away from $\mathcal{C}_0$, and hence “it lives outside the slow dynamics” \cite{SIRS}.
\end{remark}

\subsubsection{Unified formulation} \label{sec:unified_auton}

We are now ready to completely characterize the behavior of the orbits of \eqref{eq:model_auton}. Standard perturbation theory \cite[Corollary 3.1.7]{2} implies that an orbit of the perturbed system \eqref{eq:model_auton}, away from the critical manifold $\mathcal{C}_0$, follows $\mathcal{O}(\varepsilon)$-closely the orbit of the fast system \eqref{eq:fast_model_auton}, related to the same initial conditions, for $\mathcal{O}(1)$ times $t$. Even if under the flow of \eqref{eq:fast_model_auton} $z$ converges to zero, it is not obvious that \emph{in the perturbed flow} $z$ enters a $\mathcal{O}(\varepsilon^2)$-neighborhood of $\mathcal{C}_0$, starting the slow dynamics. Indeed, $z$ decreases at most exponentially, therefore it requires a time $t$ of order at least $|\log \varepsilon | \gg 1$ to enter such neighborhood \cite{2}. However, any orbit necessarily reaches the $y$-nullcline \eqref{eq:y-nullcline} (which is $\mathcal{O}(\varepsilon)$-close to $\mathcal{C}_0$) for some $y<0$ (recall Proposition \ref{prop:Gamma}). The following lemma describes the transition from the fast dynamics to the slow dynamics. 

\begin{lemma} \label{lemma:transition}
    Consider \eqref{eq:model_auton} and an initial condition $(y_*, z_*)$ such that $-b<y_*<-K_1$, $0<K_1 \in \mathcal{O}(1)$, and $z_* < \varepsilon \frac{\Bar{u}}{y_*+b}$ \eqref{eq:y-nullcline}. Let $0<K_2<K_1$ such that $K_2, K_1-K_2 \in \mathcal{O}(1)$ and denote with $(y^*, z^*)$ the point where the corresponding trajectory intersects the line $\{y=- K_2\}$. Then, for sufficiently small $\varepsilon$, we have that $z^*$ is exponentially small in $\varepsilon$. 
\end{lemma}

\begin{proof}
    Note that $\dot z \le -K_2 z$. The idea is that if $z$ becomes exponentially small in $\varepsilon$ in a time not larger than the one needed by $y$ to attain the value $y^*=- K_2$, then the result follows. Indeed, during this process $y$ increases because $z < \varepsilon \frac{\Bar{u}}{y+b}$. Since $\dot{y} \le \varepsilon \Bar{u}$, then $y$ takes a time $t$ of order at least equal to $1/\varepsilon$ to reach the value $y^*$, which is the time that a function behaving like $z$, meaning a function that decreases exponentially, needs to become exponentially small in $\varepsilon$. 
\end{proof} 

Under the hypothesis of the lemma, the value $y_\infty$ provided by \eqref{eq:u_infty} represents (at worst) an $\mathcal{O}(\varepsilon |\log \varepsilon|)$-approximation of the actual value attained by $y$ when the orbit enters the slow dynamics. Indeed, 
since $z$ decreases exponentially, an orbit takes a time $t$ of order $|\log \varepsilon|$ to travel from $(y_*, z_*)$ to an $\mathcal{O}(\varepsilon^2)$-neighborhood of $\mathcal{C}_0$. However, $y_* \le y(t) \le y_* + \varepsilon t = y_* + \mathcal{O}(\varepsilon |\log \varepsilon|)$. Moreover, note that $y_*=y_\infty + \mathcal{O}(\varepsilon)$. Therefore, in the remainder of the section we will denote all these quantities as $y_\infty$, recalling that they coincide in the limit $\varepsilon \to 0$. 

We can finally consider the compositions $\Pi_{\textnormal{fast}} \circ \Pi_\textnormal{slow}$ and $\Pi_{\textnormal{slow}} \circ \Pi_\textnormal{fast}$ defining two sequences of entry points and exit point, respectively: 
\begin{equation}
\begin{split}
    y_{\infty, 1} \coloneqq y_\infty, \quad y_{\infty, n} \coloneqq & (\Pi_\textnormal{fast} \circ \Pi_\textnormal{slow})(y_{\infty, n-1}) \\  =&\Pi_\textnormal{fast}(-y_{\infty, n-1})
\end{split}
\end{equation}
and 
\begin{equation}
\begin{split}
    y_{E, 1} \coloneqq \Pi_\textnormal{slow} (y_\infty), \quad y_{E, n} \coloneqq & (\Pi_\textnormal{slow} \circ \Pi_\textnormal{fast})(y_{E, n-1}) \\ 
    = & -\Pi_\textnormal{fast}(y_{E, n-1})
\end{split}
\end{equation}
with $n \ge 2$. It is easy to demonstrate that the former is increasing, while the latter is decreasing, and that both of them converge to $0$ \cite{jardon2021geometric}. In other words, 
\begin{equation} \label{eq:pi_fast_aux}
    \Pi_\textnormal{fast}(y)>-y.
\end{equation}
However, when $t \to + \infty$, the described subsequent entry-exit processes stop occurring before those sequences converge to $0$. Indeed, the hypotheses of Lemma \ref{lemma:transition} cannot be always satisfied because $y_{\infty,n}$ would not be $\mathcal{O}(1)$-far from $0$ \cite{ jardon2021geometric, allee}. However, at this point the linearization around the equilibrium $\mathbf{e}$ \eqref{eq:equilib} tells us that the orbit is attracted by such equilibrium, ending the entry-exit process. In conclusion, $\mathbf{e}$ is globally asymptotically stable. This behavior is represented in Figure \ref{fig:convergence}. 

\begin{figure}[h!]
\centering
  \begin{tikzpicture}
 \node at (0,0) {\includegraphics[width=.4165\linewidth]{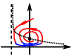}};
 \node at (-2.8,-2) {$-b$};
 \node at (3,-2) {$y$};
 \node at (-1,2.2) {$z$};
 \node at (-0.5,-0.7) {$\mathbf{e}$};
  \end{tikzpicture}
\caption{Dynamics of system \eqref{eq:model_auton}. Slow parts of the orbit are depicted in blue and with a single arrow along them, fast parts in red and with a double arrow. The orbits are attracted to the focus $\mathbf{e}$ \eqref{eq:equilib} after a finite number of entry-exit phenomena. The dotted black curve represents the $y$-nullcline \eqref{eq:y-nullcline}. \label{fig:convergence}}\end{figure}
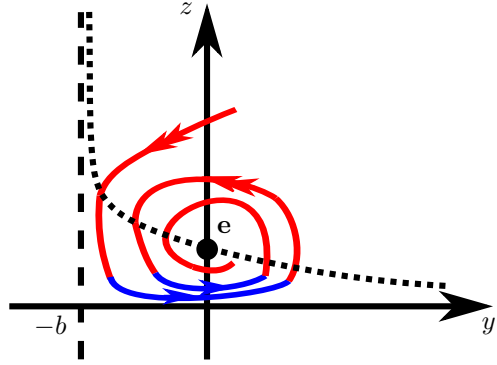

\begin{remark}
    Even if the orbits evolve inside the unbounded set $\Delta$ \eqref{eq:Delta}, in practice they are limited by $\Gamma$ \eqref{eq:Gamma} and $\Pi_{\textnormal{slow}}$ \eqref{eq:map_Pi_slow}. Therefore, $z$ and $y$ always attain at most $\mathcal{O}(1)$-values. This, in particular, implies that $y_\infty$ must be $\mathcal{O}(1)$-far from $-b$ (recall \eqref{eq:u_infty}) and that the orbits intersect the $y$-nullcline \eqref{eq:y-nullcline} when it is $\mathcal{O}(\varepsilon)$-close to the critical manifold $\mathcal{C}_0$. 
\end{remark}

\section{Non-autonomous system} \label{sec:non-auton} 

\subsection{Multi-timescale analysis} 

Analogously to Section \ref{sec:multi_auton}, we analyze the model's fast and slow dynamics and their combination.

\subsubsection{Fast formulation} 

Since in system \eqref{eq:model} the quantity $u(t)$ only appears in the product $\varepsilon u(t)$, one expects that the fast flow is not affected by the possible variations over time of $u(t)$. Therefore, the results derived in Section \ref{sec:fast_auton} for $u(t) \equiv \Bar{u} \in [u_m, u_M]$ should hold also for the non-autonomous versions of the system. In particular, the critical manifold $\mathcal{C}_0$ does not vary over time and coincides with \eqref{eq:critical_man}. 

This reasoning can be made rigorous by introducing the variable $s=t$ and writing \eqref{eq:model} as 
\begin{align} \label{eq:model_third_variable}
\begin{aligned}
    \dot y & = \varepsilon u(s) - z(y+b), \\
    \dot z & = zy, \\
    \dot s & = 1, 
\end{aligned}
\end{align}
which is an autonomous system. Indeed, after setting $\varepsilon=0$, the evolution of $(y,z)$ coincides with \eqref{eq:fast_model_auton}. 

This implies that we can consider the quantity $u(t)$ only when the orbit is $\mathcal{O}(\varepsilon)$-close to $\mathcal{C}_0$, i.e., when it evolves on the slow timescale. This allows us to define $U(\tau) \coloneqq u(t)$ and analyze the non-autonomous slow dynamics using this new quantity. 

\subsubsection{Slow formulation} 

Assume that an orbit enters a $\mathcal{O}(\varepsilon)$-neighborhood of the critical manifold $\mathcal{C}_0$ for some value of $y=y_\infty \in (-b,0)$, and once again rescale $z=\varepsilon x$. Then, on $\mathcal{C}_0$ 
\begin{equation} \label{eq:slow_flow_t}
    y' = U(\tau),  
\end{equation}
while, more generally, as long as $x$ is $\mathcal{O}(\varepsilon)$-small, 
\begin{equation} \label{eq:slow_flow_t_aux}
    y' = U(\tau) - x(y+b) = U(\tau) + \mathcal{O}(\varepsilon). 
\end{equation}
Therefore, $y$ is still a strictly increasing function (recall that, from our assumptions, $u$ may only assume strictly positive $\mathcal{O}(1)$ values), implying that the orbit necessarily travels in both the attractive and the repelling regions of $\mathcal{C}_0$, leading to an entry-exit process. During this process, starting from $x =\Tilde{x} \in \mathcal{O}(\varepsilon)$, $x$ remains $\mathcal{O}(\varepsilon)$-small, hence 
\begin{equation}
    \varepsilon \frac{\dd}{\dd \tau} \log x(\tau) = y.
\end{equation}
Therefore, for any $\tau \in (0, \tau_E)$ we have 
\begin{equation}
    \varepsilon \log \frac{x(\tau)}{\Tilde{x}} = \int_0^\tau y(s) \dd s. 
\end{equation}
Hence, the expression for the (slow) exit time $\tau_E$ \eqref{eq:exit_time_auton} directly follows. Moreover, it is possible to deduce the exit point \emph{a posteriori} after computing $\tau_E$, indeed from \eqref{eq:slow_flow_t} and \eqref{eq:slow_flow_t_aux} it follows that 
\begin{equation} \label{eq:exit_point_t}
    y_E = y_\infty + \int_0^{\tau_E} \left(U(\tau)+ \mathcal{O}(\varepsilon)\right) \dd \tau = y_\infty + \int_0^{\tau_E} U(\tau) \dd \tau
\end{equation}
in the limit $\varepsilon \to 0$. 

\begin{remark}
Due to the simple structure of our toy model, Eq. \eqref{eq:exit_time_auton} provides the exact expression of the exit time, regardless of the value of $\varepsilon$. On the other hand, Eqs. \eqref{eq:exit_point_auton} and \eqref{eq:exit_point_t} provide an $\mathcal{O}(\varepsilon)$-approximation of the exit point. 
\end{remark}

A natural problem is finding the expression of $U(\tau)$ that maximizes or minimizes the exit point. Recall \eqref{eq:exit_time_auton}; since $\dd y = U(\tau) \dd \tau$, then (with a slight abuse of notation motivated by the fact that $y(\tau)$ is strictly increasing and hence invertible)
\begin{equation}
    \int_{y_\infty}^{y_E} \frac{y}{U(y)} \dd y=0.
\end{equation} 
Therefore, it is easy to see that 
\begin{align} \label{eq:u_max_yE}
    U(y)=
    \begin{cases}
        u_m \quad & \textnormal{if } y < 0, \\
        u_M \quad & \textnormal{if } y > 0, \\
    \end{cases}
\end{align}
maximizes $y_E$, while by reversing $u_m$ and $u_M$ in \eqref{eq:u_max_yE} we obtain the expression of $U(\tau)$ that minimizes $y_E$. Recall \eqref{eq:map_Pi_slow}, in the limit $\varepsilon \to 0$ we can therefore define the following:  
\begin{equation} \label{eq:map_Pi_slow_max}
    \Pi_{\textnormal{slow}}^{\max} \colon \{y \in (-b, 0)\} \to \left\{ y \in \left(0, b \sqrt{\frac{u_M}{u_m}}\right) \right\}
\end{equation}
that maps $y_\infty$ to the maximum possible exit point $y_E^{\max} = -y_\infty \sqrt{\frac{u_M}{u_m}}$, and 
\begin{equation} \label{eq:map_Pi_slow_min}
    \Pi_{\textnormal{slow}}^{\min} \colon \{y \in (-b, 0)\} \to \left\{ y \in \left(0, b \sqrt{\frac{u_m}{u_M}}\right) \right\}
\end{equation}
that maps $y_\infty$ to the minimum possible exit point $y_E^{\min}= -y_\infty \sqrt{\frac{u_m}{u_M}}$. The exit time related to $\Pi_{\textnormal{slow}}^{\max}$ is $-y_\infty\left(\frac{1}{u_m} + \frac{1}{\sqrt{u_m u_M}}\right)$, while the exit time related to $\Pi_{\textnormal{slow}}^{\min}$ is $-y_\infty\left(\frac{1}{u_M} + \frac{1}{\sqrt{u_m u_M}}\right)$ (which is trivially smaller than the former, as $u_M>u_m$).

All possible exit points related to $y_\infty$ are bounded between $y_E^{\min}$ and $y_E^{\max}$. Given one of them, a natural question is which choice of $U(\tau)$ leads to it in the minimum possible exit time. 

\begin{theorem} \label{teo:minimize}
    Consider the entry point $y_\infty \in (-b, 0)$ and an exit point $y_E \in [-y_\infty, y_E^{\max}]=[-y_\infty, -y_\infty \sqrt{u_M/u_m}]$ (resp. $y_E \in [y_E^{\min}, -y_\infty] = [-y_\infty \sqrt{u_m/u_M}, -y_\infty]$), then 
    \begin{align} \label{eq:optimal_time}
        U(y)= 
        \begin{cases}
            u_m \quad &\textnormal{if } y \in [y_\infty, y_s], \\
            u_M \quad &\textnormal{if } y \in (y_s, y_E], \\
        \end{cases}
        \quad \left(resp. \;\; U(y)= 
        \begin{cases}
            u_M \quad &\textnormal{if } y \in [y_\infty, y_s], \\
            u_m \quad &\textnormal{if } y \in (y_s, y_E], \\
        \end{cases}\right)
    \end{align}
    with  
    \begin{equation} \label{eq:switching_point}
        y_s = - \sqrt{\frac{u_M y_\infty^2 - u_m y_E^2}{u_M - u_m}} \quad \left(resp. \;\; y_s = \sqrt{\frac{u_M y_E^2 - u_m y_\infty^2}{u_M - u_m}} \right)
    \end{equation}
    leads to $y_E$ in the minimum possible exit time 
    \begin{equation} \label{eq:exit_time_minimum}
        \tau_E = \frac{y_s - y_\infty}{u_m} +\frac{y_E - y_s}{u_M} \quad \left(resp. \;\; \tau_E = \frac{y_s - y_\infty}{u_M} +\frac{y_E - y_s}{u_m}\right).
    \end{equation}
\end{theorem}

\begin{proof}
    Set $v(y)=1/U(y)$. Recall \eqref{eq:slow_flow_t_aux}; we want to minimize
    \begin{equation}
        J[v] \coloneqq \int_{y_\infty}^{y_E} v(y) \dd y = \tau_E
    \end{equation} 
    under the constraint 
    \begin{equation} \label{eq:contraint}
        \int_{y_\infty}^{y_E} y v(y) \dd y = 0. 
    \end{equation}
    Consider the convex set of controls 
    \begin{equation}
        \mathcal{C} \coloneqq \{v \colon [y_\infty, y_E] \to [1/u_M, 1/u_m] \; | \; v \; \textnormal{Lebesgue measurable}\}
    \end{equation}
    and the set of functions that satisfy the previous constraint 
    \begin{equation}
        \mathcal{H} \coloneqq \{v \colon [y_\infty, y_E] \to [1/u_M, 1/u_m] \; | \;\textnormal{\eqref{eq:contraint} holds}\}.
    \end{equation}
    Then, the convex set of admissible controls is $\mathcal{A}\coloneqq \mathcal{C} \cap \mathcal{H}$. Eq. \eqref{eq:contraint} implies that, for any $v \in \mathcal{A}$ and $\mu \in \mathbb{R}$, 
    \begin{equation}
        J[v] = \int_{y_\infty}^{y_E} \left(1 + \mu y\right) v(y) \dd y \eqqcolon J_\mu [v].
    \end{equation}
    In other words, over $\mathcal{A}$, $J \equiv J_\mu$, while over $\mathcal{C} \setminus \mathcal{A}$, $J \not\equiv J_\mu$. This implies that minimizing $J$ over $\mathcal{A}$ is equivalent to minimizing $J_\mu$ over $\mathcal{A}$, for any $\mu \in \mathbb{R}$. Fixed $\mu$, the function $v_\mu \in \mathcal{C}$ that minimizes $J_\mu$ over $\mathcal{C}$ is  
    \begin{align}
        v_\mu(y) = 
        \begin{cases}
            1/u_M \quad &\textnormal{if } 1+\mu y > 0, \\ 
            1/u_m & \textnormal{if } 1+\mu y < 0. 
        \end{cases}
    \end{align}
    Therefore, $J_\mu[v] \ge J_\mu[v_\mu]$ for any $v \in \mathcal{A} \subset \mathcal{C}$. If we manage to find $\Tilde{\mu}$ such that $v_{\Tilde{\mu}} \in \mathcal{A}$ (i.e., such that it satisfies \eqref{eq:contraint}), then $v_{\Tilde{\mu}}$ would minimize $J_{\Tilde{\mu}}$ over both $\mathcal{C}$ and $\mathcal{A}$; since $J \equiv J_{\Tilde{\mu}}$ over $\mathcal{A}$ it also minimizes $J$ over $\mathcal{A}$. If $y_E \in (y_E^{\min}, y_E^{\max})$, trivial calculations allow us to derive \eqref{eq:switching_point} and \eqref{eq:exit_time_minimum}; in particular, the optimal control is unique almost everywhere. The cases $y_E=y_E^{\min}$ and $y_E=y_E^{\max}$ formally correspond to $\Tilde{\mu}=-\infty$ and $\Tilde{\mu}=+\infty$, respectively; \eqref{eq:switching_point} and \eqref{eq:exit_time_minimum} follow from the fact that they are the unique controls (up to negligible sets) that allow for reaching these extreme exit points. 
\end{proof} 

Direct, cumbersome calculations (not shown here for brevity) demonstrate that, among all acceptable exit points $y_E \in [y_E^{\min}, y_E^{\max}]$, the fastest to reach with an optimal strategy is $-y_\infty$. Moreover, as $|y_E - (-y_\infty)|$ increases, such optimal (slow) time increases. Also, denote with $\tau(y_\infty, \lambda)$ the slow time needed to travel from $y_\infty$ to $y_E=-\lambda y_\infty$ according to the optimal strategy presented in Theorem \ref{teo:minimize}. Given two entry points $y_{\infty,1}, y_{\infty,2}$ and an exit point $y_E=-\lambda_1 y_{\infty,1}=-\lambda_2y_{\infty,2}$, if $|y_{\infty,1}|<|y_{\infty,2}|<y_E$ then $\tau(y_{\infty,1}, \lambda_1)>\tau(y_{\infty,2}, \lambda_2)$. 

Finally, notice that 
\begin{equation}
    \tau(y_\infty, \lambda) = -y_\infty \phi(\lambda), 
\end{equation}
where the (cumbersome) expression of $\phi$ can be directly derived from \eqref{eq:switching_point}--\eqref{eq:exit_time_minimum}. In the following, we just need to know that there exists some $\phi_{\max}>\phi_{\min}>0$ such that $\phi_{\min} \le \phi(\lambda) \le \phi_{\max}$. 

\subsubsection{Unified formulation} \label{sec:unified}

Similarly to the autonomous scenario, the value $y_\infty$ provided by \eqref{eq:u_infty} represents (at worst) an $\mathcal{O}(\varepsilon |\log \varepsilon|)$-approximation of the actual value attained by $y$ when the orbit enters the slow dynamics. Indeed, since $U(\tau) \in \mathcal{O}(1)$, when $y<0$ and $z$ is small, it is straightforward to bound $\dot{y}$ to show that $z$ becomes $\mathcal{O}(\varepsilon^2)$-small exponentially fast. After computing the exit time according to \eqref{eq:exit_time_auton}, Eq. \eqref{eq:exit_point_t} provides the exit point. However, opposite to what happens in the autonomous scenario, in general we do not have an increasing sequence of entry points and a decreasing sequence of exit points. 

A natural question is if, by controlling $U(\tau)$, is it possible to generate an orbit that describes a (stable) limit cycle? Such cycle must start the slow dynamics for some $y_\infty^\textnormal{cycle} \in (-b,0)$ and exit from it for some $y_E^\textnormal{cycle}>0$. To this end, we use the map $\Pi_{\textnormal{slow}}^{\max}$ \eqref{eq:map_Pi_slow_max}, since it is clearly the best candidate to prevent the decrease in the exit points observed in the autonomous setting. The following proposition ensures the existence of such orbit, regardless of the values attained by the parameters. 

\begin{proposition} \label{prop:cycle}
    The composition of maps $\Pi_{\textnormal{fast}} \circ \Pi_{\textnormal{slow}}^{\max}$ (or $ \Pi_{\textnormal{slow}}^{\max} \circ \Pi_{\textnormal{fast}}$) describes a stable limit cycle in the limit $\varepsilon \to 0$. In particular, such cycle enters the slow dynamics for some $y_\infty^{\textnormal{cycle}} \in \left(-b, -b\left(1-\sqrt{\frac{u_m}{u_M}}\right)\right)$. 
\end{proposition}

\begin{proof}
    A limit cycle exists if $\Pi_{\textnormal{fast}} \circ \Pi_{\textnormal{slow}}^{\max}$ admits a fixed point, i.e., if there exists $y_\infty^{\textnormal{cycle}} \in (-b, 0)$ such that (recall \eqref{eq:riscrit} and \eqref{eq:map_Pi_slow_max})
    \begin{equation} \label{eq:cycle}
        b \log \left(\frac{b+y_\infty^{\textnormal{cycle}}}{b-y_\infty^{\textnormal{cycle}} \sqrt{\frac{u_M}{u_m}}}\right) - y_\infty^{\textnormal{cycle}} \left(1+\sqrt{\frac{u_M}{u_m}}\right) = 0. 
    \end{equation}
    Define $F(x)=b \log \left(\frac{b+x}{b-x \sqrt{u_M/u_m}}\right) - x (1+\sqrt{u_M/u_m})$ and note that $\displaystyle\lim_{x \to -b}F(x)=-\infty$, while $F(0)=0$. Since $F$ increases in $(-b, -b(1-\sqrt{u_m/u_M}))$ and decreases in $(-b(1-\sqrt{u_m/u_M}), 0)$ there exists a unique solution to \eqref{eq:cycle} and it lies in $(-b, -b(1-\sqrt{u_m/u_M}))$. Finally, the stability of the limit cycle follows from the monotony of the sequence of entry points described by $\Pi_{\textnormal{fast}} \circ \Pi_{\textnormal{slow}}^{\max}$. 
\end{proof}

Note that different choices of $U(\tau)$ can lead to the presence of different limit cycles. However, the one described by $\Pi_{\textnormal{fast}} \circ \Pi_{\textnormal{slow}}^{\max}$ has the maximum possible amplitude. Moreover, we remark that we are not interested in the solution $y=0$ of \eqref{eq:cycle}, since it does not define a cycle through the combination of the maps defined above, which is our aim with Proposition \ref{prop:cycle}.

\section{Minimum-time optimal control problems} \label{sec:opt_cont}

We are now ready to apply the previous results to solve minimum-time problems. Namely, we want to find the optimal function $U(\tau)$ such that an orbit reaches a target value $0<\hat{z} \in \mathcal{O}(1)$ in the minimum time possible. Since these optimal control problems do not necessarily have a solution, we divide our analysis into two part: establishing if it is possible to reach $\hat{z}$ and, if it is, determining the optimal control $U(\tau)$ to do so in the minimum time. 

We remark that the problem of reaching some $\hat{y}$ can be treated in an identical way, since it suffices to associate $\hat{z}$ to it through $\Gamma$ \eqref{eq:Gamma}. 

\subsection{Admissibility of the target} \label{sec:admissibility}

Consider an orbit starting from $(y_0, z_0)$ in the fast flow; it follows $\Gamma$ \eqref{eq:Gamma} and enters the slow dynamics for some $y_{\infty,1} \in (-b,0)$ given by \eqref{eq:u_infty} in the limit $\varepsilon \to 0$. If, during this initial fast flow, which is independent on $u(t)$, $z$ attains the value $\hat z$, then the target is reached and there is an infinite number of optimal solutions. Suppose that this does not happen, then the non-autonomous entry-exit process previously described begins. 

Note that the larger the value of $y_{E,1}$, the larger will be the largest value attained by $z$ during the subsequent fast flow described by $\Gamma(y,z)=\Gamma(y_{E,1}, 0)$ \eqref{eq:Gamma}. Moreover, larger values of $y_{E,1}$ imply smaller values of $y_{\infty,2}$, which in turn can lead to $y_{E,2}>y_{E,1}$, and so on. Therefore, in order to understand whether it is possible or not to reach $\hat{z}$, we want to maximize $y_{E,1}$ and, eventually, repeat this process. To this end, we should consider the composition of maps $\Pi_{\textnormal{fast}} \circ \Pi_{\textnormal{slow}}^{\max}$ (\eqref{eq:map_Pi_fast} and \eqref{eq:map_Pi_slow_max}). Recall that Proposition \ref{prop:cycle} ensures that this strategy creates a stable limit cycle that enters the slow for some $y_\infty^{\textnormal{cycle}} \in \big(-b, -b(1-\sqrt{u_m/u_M})\big)$ given by \eqref{eq:cycle} and exits from it for $y_E^{\textnormal{cycle}} = -y_\infty^\textnormal{cycle} \sqrt{u_M/u_m}$. 

If $y_{\infty,1} < y_\infty^\textnormal{cycle}$, then the sequence of exit points $\{y_{E,n}\}=\{\sqrt{u_M/u_m} y_{\infty,n}\}$ is decreasing, therefore $z$ attains its largest value during the first fast flow. Recall $\Gamma$ \eqref{eq:Gamma}; such largest value of $z$ is attained when $y=0$ and it is equal to 
\begin{equation}
    b \log \left(\frac{b}{b - \sqrt{\frac{u_M}{u_m}} y_{\infty,1}}\right) - \sqrt{\frac{u_M}{u_m}} y_{\infty,1}. 
\end{equation}
This situation corresponds to an orbit evolving “outside” the limit cycle. 

On the other hand, if $y_\infty > y_\infty^\textnormal{cycle}$, then the sequence of exit points $\{y_{E,n}\}=\{\sqrt{u_M/u_m} y_{\infty,n}\}$ is increasing and converging to $y_E^\textnormal{cycle}$. Since the orbit converges toward the limit cycle, $z$ is bounded by its maximum value attained on the cycle. Again, such maximum value of $z$ is reached for $y=0$ and it is equal to 
\begin{equation}
    b \log \left(\frac{b}{b + y_E^\textnormal{cycle}}\right) + y_E^\textnormal{cycle}. 
\end{equation}
This situation corresponds to an orbit evolving “inside” the limit cycle. 

In conclusion, we obtain the following lemma (see Figure \ref{fig:y_stop} for a visualization). 

\begin{lemma} \label{lemma:stop}
    Define 
    \begin{equation} \label{eq:y_stop}
        y^\textnormal{stop} \coloneqq -b - b W_0\left(-\frac{1}{b}\exp\left(-\frac{\hat z + b - b \log b}{b}\right)\right) \in (-b,0), 
    \end{equation} 
    where $W_0$ denotes the principal branch of the Lambert $W$ function, and consider the limit $\varepsilon \to 0$. If $y_\infty^\textnormal{cycle} < y^\textnormal{stop}$, then it is possible to reach $\hat{z}$ starting from any $y_{\infty,1} \in (-b, 0)$. Otherwise, if $y_\infty^\textnormal{cycle} \ge y^\textnormal{stop}$, then it is possible to reach $\hat{z}$ if and only if 
    \begin{equation} \label{eq:y_infty_y_stop}
        y_{\infty,1} \le - \sqrt{\frac{u_m}{u_M}}\Pi_\textnormal{fast}^{-1}(y^\textnormal{stop}) \in (-b, y^\textnormal{stop}]. 
    \end{equation}
\end{lemma}
\begin{proof}
    Recall $\Gamma$ \eqref{eq:Gamma}, an orbit that travels in the fast flow from the repelling region part of $\mathcal{C}_0$ to its attractive region reaches the value $\hat z$ if and only if, in the limit $\varepsilon \to 0$, it enters the slow dynamics for $y_\infty \le y^\textnormal{stop}$ \eqref{eq:y_stop}. Therefore, if $y_\infty^\textnormal{cycle} < y^\textnormal{stop}$, then the maximum value attained by $z$ on the limit cycle is larger than $\hat{z}$. Hence, it is possible to reach $\hat{z}$ starting from any $y_{\infty,1} \in (-b, 0)$. On the other hand, if $y_\infty^\textnormal{cycle} \ge y^\textnormal{stop}$, then it is possible to reach $\hat{z}$ only if the entry-exit processes start from some $y_{\infty,1} < y_\infty^\textnormal{cycle}$ satisfying \eqref{eq:y_infty_y_stop}, in this case $\hat{z}$ is reached during the first fast flow (starting from $y_{E,1}=-\sqrt{u_M/u_m} y_{\infty,1}=\Pi_\textnormal{slow}^{\max}(y_{\infty,1})$).  
\end{proof}

\begin{figure}[h!]
    \centering
\begin{subfigure}{.45\textwidth}
  \centering   
\begin{tikzpicture}
 \node at (0,0) {\includegraphics[width=.8\linewidth]{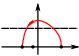}};
 \node at (1.5,-1.8) {$y_E^\textnormal{cycle}$};
  \node at (-1.4,-1.8) {$y_{\infty}^\textnormal{cycle}$};
\node at (-0.55,-1.1) {$y^\textnormal{stop}$};
 \node at (2,0.25) {$z=\hat{z}$};
  \end{tikzpicture}
  \caption{}
\end{subfigure}\hspace{0.5cm}
\begin{subfigure}{.45\textwidth}
  \centering
\begin{tikzpicture}
 \node at (0,0) {\includegraphics[width=.8\linewidth]{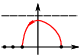}};
 \node at (1.5,-1.8) {$y_E^\textnormal{cycle}$};
  \node at (-1.4,-1.8) {$y_{\infty}^\textnormal{cycle}$};
 \node at (-2.7,-1.8) {$\Tilde{y}$};
 \node at (2,1.25) {$z=\hat{z}$};
\node at (-2,-1.1) {$y^\textnormal{stop}$};
  \end{tikzpicture}
  \caption{}
\end{subfigure}
    \caption{Illustration of Lemma \ref{lemma:stop}, we set $\Tilde{y}=- \sqrt{u_m/u_M}\Pi_\textnormal{fast}^{-1}(y^\textnormal{stop})$. Note the position of the target $\hat{z}$ with respect to the red curve, which describes the fast flow of the cycle obtained in the limit $\varepsilon \to 0$ (recall Proposition \ref{prop:cycle}). (a) $y_\infty^\textnormal{cycle} < y^\textnormal{stop}$, hence it is possible to reach $\hat{z}$ starting from any $y_{\infty,1}$. (b) $y_\infty^\textnormal{cycle} \ge  y^\textnormal{stop}$, therefore it is possible to reach $\hat{z}$ only if $y_{\infty,1} \le \Tilde{y} \le y^\textnormal{stop} \le y_\infty^\textnormal{cycle}$.  \label{fig:y_stop} }
\end{figure}
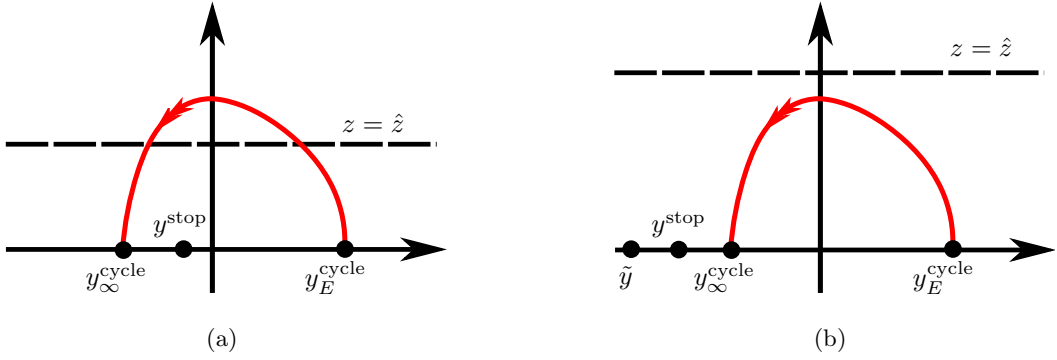

\subsection{Derivation of the optimal control}

Assume that it is possible to reach $\hat z$ starting from some $y_{\infty,1} \in (-b,0)$ (recall Lemma \ref{lemma:stop}); we start by highlighting two fundamental facts. Firstly, in the limit $\varepsilon \to 0$, we can restrict ourselves to minimizing the total time spent in the slow timescale, since it is of the order $\mathcal{O}(1/\varepsilon)$ (if measured with respect to the fast time variable $t$). Indeed, the time of the fast flow is of the order $\mathcal{O}(1)$, while the time needed for the transition from one dynamic to the other is of the order $\mathcal{O}(|\log \varepsilon|)$. Secondly, if an orbit enters the slow dynamics for some $y_\infty$ and we want it to reach some $y_E \in [y_E^\textnormal{min}, y_E^\textnormal{max}]$ (recall \eqref{eq:map_Pi_slow_max} and \eqref{eq:map_Pi_slow_min}) in the minimum possible time, the best strategy may not necessarily be the one described in Theorem \ref{teo:minimize}. Indeed, it is possible (and we will show this numerically in Section \ref{sec:example}) that it is convenient to reach $y_E$ by combining several entry-exit processes. However, clearly, for each of those slow dynamics we must choose $U(\tau)$ according to Theorem \ref{teo:minimize}.

\subsubsection{Formulation of the problem} \label{sec:formulation}

Consider an orbit that starts in the slow timescale for some $y_{\infty,1} \in (-b,0)$; the admissible exit points are of the form $- \lambda_1 y_{\infty,1}$ with $\lambda_1 \in [\sqrt{u_m/u_M}, \sqrt{u_M/u_m}]$ (recall \eqref{eq:map_Pi_slow_max} and \eqref{eq:map_Pi_slow_min}). We can therefore define a sequence of exit points and entry points as (recall \eqref{eq:map_Pi_fast}) 
\begin{equation}
    y_{E,n} = -\lambda_n y_{\infty,n} \quad \textnormal{and} \quad y_{\infty, n+1} = \Pi_\textnormal{fast}(y_{E,n}), 
\end{equation}
for $n \ge 1$, where the $\lambda_n \in [\sqrt{u_m/u_M}, \sqrt{u_M/u_m}]$ completely determine the dynamics. Denote with $\tau(y_{\infty,n}, \lambda_n)$ the slow time needed to travel from $y_{\infty,n}$ to $y_{E,n}$ according to the optimal strategy presented in Theorem \ref{teo:minimize}. Then, we want to find the slow time $T(y_{\infty,1})$ given by
\begin{equation} \label{eq:T_def}
    T(y_{\infty,1}) \coloneqq \inf_{\{\lambda_n\}} \sum_{n=1}^{+\infty} \beta_n \tau(y_{\infty,n}, \lambda_n), 
\end{equation}
where $\{\lambda_n\}$ is chosen so that there exists $\Bar{n} \ge 1$ such that $y_{\infty,\Bar{n}+1}\le y^\textnormal{stop}$ \eqref{eq:y_stop} (i.e., the orbit reaches the target $\hat z$ in the fast flow) and $\beta_n=0$ for all $n \ge \Bar{n}+1 \ge 2$, otherwise $\beta_n=1$. Assuming that it is possible to reach $\hat z$ starting from $y_{\infty,1}$, we want to demonstrate that the infimum \eqref{eq:T_def} is a minimum, i.e., that an optimal control exists. Clearly, $T(y_{\infty,1})>0$. 

\begin{remark} \label{rmk:infinite_sum}
    Choosing $\{\lambda_n\}$ such that there exists $\Bar{n} \ge 1$ such that $y_{\infty,\Bar{n}+1}\le y^\textnormal{stop}$ is fundamental for two reasons. First, if it is not possible to reach $\hat{z}$ starting from $y_{\infty,1}$, then we correctly obtain $T(y_{\infty,1})=+\infty$ because we are calculating the infimum of an empty set. Second, for a strategy $\{\lambda_n\}$ that does not lead to $\hat{z}$, it could be possible (see Appendix \ref{appA}) that $\sum_{n=1}^{+\infty} \beta_n \tau(y_{\infty,n}, \lambda_n) = \sum_{n=1}^{+\infty} \tau(y_{\infty,n}, \lambda_n) < +\infty$; therefore, the minimum might be related to a non-acceptable strategy, instead of one leading to the target. 
\end{remark} 

\subsubsection{Existence and construction of an optimal solution} \label{sec:existence}

The optimal solution is not necessarily the one characterized by the smallest number of entry-exit phenomena. However, the following result holds. 

\begin{lemma} \label{lemma:N_finite}
    Assume that it is possible to reach $\hat z$ starting from $y_{\infty,1}$. Then, there exists a (generally non-unique) strategy $\{\lambda_n\}$ that leads to $\hat z$ in the minimum possible number $N(y_{\infty,1}) < +\infty$ of entry-exit phenomena, therefore $T(y_{\infty,1})<+\infty$. In particular, if $y_{\infty,1} \le y_\infty^\textnormal{cycle}$ then $N(y_{\infty,1})=1$. 
\end{lemma}

\begin{proof}
    This follows directly from the considerations made to establish the possibility or non-possibility of reaching $\hat z$ in Section \ref{sec:admissibility}. The finite minimum value $N(y_{\infty,1})$ is related to the strategy $\Pi_\textnormal{fast} \circ \Pi_\textnormal{slow}^{\max}$. The non-uniqueness is trivial since small ``tweaking'' of such a strategy may still lead to the desired result in the same number of entry-exit phenomena.
\end{proof} 

We derive now a lower bound for each exit time related to strategies that could be optimal. 

\begin{lemma} \label{lemma:tau_min}
    Assume that it is possible to reach $\hat z$ starting from $y_{\infty,1}$, then there exists $\tau_{\min}=\tau_{\min}(y_{\infty,1})>0$ such that every strategy $\{\lambda_n\}$ leading to $\hat z$ and containing an entry-exit process of duration smaller than $\tau_{\min}$ cannot have total time arbitrarily close to $T(y_{\infty,1})$. 
\end{lemma}

\begin{proof}
    Assume that $y_{\infty,1} > y^\textnormal{stop}$; note that this is not a restrictive assumption since the exit time is equal to zero only if the entry point is $0$. Define (recall the consequences of Theorem \ref{teo:minimize})
    \begin{equation} \label{eq:tau_min}
        \tau_{\min}\coloneqq-\frac{\phi_{\min} u_m}{2 u_M} y_{\infty,1}. 
    \end{equation}
    Consider a strategy $\{\lambda_n\}$ leading to the target and assume that there exists $k>1$ such that 
    \begin{equation}
        \tau_{\min} > \tau(y_{\infty,k}, \lambda_k) \ge -\phi_{\min} y_{\infty,k}, 
    \end{equation}
    then \eqref{eq:tau_min} implies that 
    \begin{equation} \label{eq:y_inf_k}
        y_{\infty,1} < \frac{u_m}{2 u_M} y_{\infty,1} < y_{\infty,k}. 
    \end{equation}
    Let $j \ge k$ the first index such that 
    \begin{equation} \label{eq:js}
        y_{\infty,j}>y_{\infty,1} \quad \textnormal{and} \quad y_{\infty, j+1} \le y_{\infty,1},
    \end{equation}
    which necessarily exists because the strategy $\{\lambda_n\}$ leads to $\hat z$. 

    Now, the idea is to build another strategy that goes directly from $y_{\infty,1}$ to $y_{\infty,j+1}$. Define 
    \begin{equation}
        \lambda^* \coloneqq \lambda_j\frac{y_{\infty,j}}{y_{\infty,1}}, 
    \end{equation}
    then $\lambda^*>1$ because $y_{\infty, j+1}>\lambda_j y_{\infty,j}$ \eqref{eq:pi_fast_aux} and \eqref{eq:js}, while 
    $\lambda^*<\lambda_j \le R$ because of \eqref{eq:js}; in other words, $\lambda^*$ is admissible. Since (recall the consequences of Theorem \ref{teo:minimize})
    \begin{equation}
        \tau(y_{\infty,1}, \lambda^*) < \tau(y_{\infty,j}, \lambda_j),
    \end{equation}
    the new strategy $\{\lambda^*, \lambda_{j+1}, \lambda_{j+2}, \dots\}$ is faster than the initial one. In particular, since the time saved is at least $\tau(y_{\infty,1}, \lambda_1) \ge -\phi_{\min} y_{\infty,1}$, the proof is completed. 
\end{proof}

The same shortcut argument also shows that, if $y_{\infty,1}\ge y^\textnormal{stop}$, every optimal strategy is such that the sequence of entry points is strictly decreasing until the target is reached.

\begin{remark} \label{rmk:tau_min}
    An exit time is equal to zero only if $y_\infty=0=y_E$. This means that the only strategies with times not bounded away from $0$ are the ones collapsing on the origin and, after getting arbitrary close to it, starting to expand to reach $\hat z$. The necessity of Lemma \ref{lemma:tau_min} is related to the presence of strategies not leading to the target but for which $\sum_{n=1}^{+\infty} \beta_n \tau(y_{\infty,n}, \lambda_n) = \sum_{n=1}^{+\infty} \tau(y_{\infty,n}, \lambda_n) < +\infty$ (recall Remark \ref{rmk:infinite_sum} and Appendix \ref{appA}). Indeed, one cannot exclude \emph{a priori} this collapse-expansion process. Had such strategies not existed, Lemma \ref{lemma:tau_min} would have been trivial, since getting arbitrarily close to $0$ would have caused such a summation to become arbitrarily large, and hence to remain far from the infimum \eqref{eq:T_def}.
\end{remark}

We are now ready to prove the existence of an optimal solution.  

\begin{theorem} \label{teo:minimum_exists}
    Assume that it is possible to reach $\hat z$ starting from $y_{\infty,1}$, then: 
    \begin{enumerate} [(i)]
        \item For any $M \ge N(y_{\infty,1})$ (Lemma \ref{lemma:N_finite}), among all possible strategies that lead to $\hat z$ in at most $M$ entry-exit processes, there exists an optimal one. 
        \item The infimum \eqref{eq:T_def} is a minimum. 
    \end{enumerate}
\end{theorem}

\begin{proof}
    \textit{(i).} Define the compact set $I_M \coloneqq [\sqrt{u_m/u_M}, \sqrt{u_M/u_m}]^M$ and 
    \begin{equation}
        \Lambda_M(y_{\infty,1}) \coloneqq \{\boldsymbol{\lambda}:=(\lambda_1, \dots, \lambda_M) \in I_M \colon \; \exists \; n \in \{1, \dots, M\} \; \textnormal{such that} \; y_{\infty,n+1} \le y^\textnormal{stop}\} \ne \emptyset. 
    \end{equation}
    Since, for any $n$, the map (applied to $y_{\infty,1}$)
    \begin{equation} \label{eq:map_aux}
        (\lambda_1, \dots, \lambda_n) \mapsto y_{\infty, n+1}
    \end{equation}
    is continuous, then
    \begin{equation}
\{\boldsymbol{\lambda} \in I_M \colon \; y_{\infty,n+1} \le y^\textnormal{stop}\}
    \end{equation}
    is closed in $I_M$. Since 
    \begin{equation}
        \Lambda_M(y_{\infty,1}) = \bigcup_{n=1}^M \{\boldsymbol{\lambda} \in I_M \colon \; y_{\infty,n+1} \le y^\textnormal{stop}\}, 
    \end{equation}
    then also $\Lambda_M(y_{\infty,1})$ is closed, and hence compact because $I_M$ is compact. 

    Define 
\begin{equation}\label{eq:U_n}
        U_0 \coloneqq I_M, \quad 
        U_n \coloneqq \{\boldsymbol{\lambda} \in I_M \colon \; y_{\infty,j}>y^\textnormal{stop} \; \textnormal{for all} \; j=2, \dots, n+1\}, \; n = 1, \dots, M-1, 
    \end{equation}
    which are open because, again, the map \eqref{eq:map_aux} is continuous. The time associated to any given strategy $\boldsymbol{\lambda} \in \Lambda_M(y_{\infty,1})$ is 
    \begin{equation}
        T_{\boldsymbol{\lambda}}^M (y_{\infty,1}) = \sum_{n=1}^M \mathbb{I}_{U_{n-1}}(\boldsymbol{\lambda})\tau(y_{\infty,n}, \lambda_n), 
    \end{equation}
    which is lower semicontinuous because $\mathbb{I}_{U_{n-1}} \ge 0$ is lower semicontinuous, since $U_{n-1}$ is open, and $\tau \ge 0$ is continuous. Note that the continuity of $\tau$ is intended with respect to $(\lambda_1, \dots, \lambda_n)$ and directly follows from the continuity of each exit time (with respect to the entry point and the exit point) and of the map $\Pi_\textnormal{fast}$ \eqref{eq:map_Pi_fast}. Therefore, since $\Lambda_M(y_{\infty,1})$ is compact, there exists $\boldsymbol{\lambda}^* \in \Lambda_M(y_{\infty,1})$ such that 
    \begin{equation}
        T^M_{\boldsymbol{\lambda}^*} = \min_{\boldsymbol{\lambda} \in \Lambda_M(y_{\infty,1})} T^M_{\boldsymbol{\lambda}} (y_{\infty,1}),
    \end{equation}
    representing the optimal strategy among the ones characterized by at most $M$ entry-exit processes. 

    \medskip 

    \noindent \textit{(ii).} Define the compact set $I \coloneqq [\sqrt{u_m/u_M}, \sqrt{u_M/u_m}]^\mathbb{N}$ and 
    \begin{equation}
        \Lambda(y_{\infty,1}) \coloneqq \{\boldsymbol{\lambda} \in I \colon \; \exists \; n \ge 1 \; \textnormal{such that} \; y_{\infty,n+1} \le y^\textnormal{stop}\} \ne \emptyset, 
    \end{equation}
    which, in this case, is not necessarily compact. Define $U_n$ as in \eqref{eq:U_n} but this time over $I$. For any $m \ge 1$, define the truncated time associated to the strategy $\boldsymbol{\lambda} \in I$ as 
    \begin{equation}
        T_{\boldsymbol{\lambda}}^m (y_{\infty,1}) = \sum_{n=1}^m \mathbb{I}_{U_{n-1}}(\boldsymbol{\lambda})\tau(y_{\infty,n}, \lambda_n),
    \end{equation}
    which is lower semicontinuous. The time associated to $\boldsymbol{\lambda}$ is 
    \begin{equation}
        T_{\boldsymbol{\lambda}} (y_{\infty,1}) = \sup_{m \ge 1} T_{\boldsymbol{\lambda}}^m (y_{\infty,1}), 
    \end{equation}
    which is lower semicontinuous, as well. 

    Define 
    \begin{equation} \label{eq:T_def_bis}
        T(y_{\infty,1}) \coloneqq \inf_{\boldsymbol{\lambda} \in \Lambda(y_{\infty,1})} T_{\boldsymbol{\lambda}} (y_{\infty,1}) < + \infty, 
    \end{equation}
    we want to prove that such infimum is in fact a minimum. Consider a sequence $\{\boldsymbol{\lambda}_k\} \subset \Lambda(y_{\infty,1})$ such that $T_{\boldsymbol{\lambda}_k} (y_{\infty,1}) \to T(y_{\infty,1})$ as $k \to + \infty$. Since $T(y_{\infty,1}) < +\infty$ (Lemma \ref{lemma:N_finite}), then there exists some $C>0$ such that $T_{\boldsymbol{\lambda}_k} (y_{\infty,1}) \le C$. Define $M(\boldsymbol{\lambda}_k)$ as the number of entry-exit processes needed by the $k$th-strategy to reach the target. Since, for $k$ sufficiently large, $\tau \ge \tau_{\min}>0$ (Lemma \ref{lemma:tau_min}), then $M(\boldsymbol{\lambda}_k) \le C/\tau_{\min}$. Therefore, $M(\boldsymbol{\lambda}_k) \in \{1, \dots, \lfloor{C/\tau_{\min}}\rfloor\}$ for any $k$ sufficiently large. This implies that there exists a value $\Tilde{M} \in \{1, \dots, \lfloor{C/\tau_{\min}}\rfloor\}$ that is attained by $M(\boldsymbol{\lambda}_k)$ an infinite number of times. Consider a (not-renamed) subsequence $\boldsymbol{\lambda}_k$ such that $M(\boldsymbol{\lambda}_k) \equiv \Tilde{M}$. Again, consider another (not-renamed) subsequence such that $\boldsymbol{\lambda}_k \to \boldsymbol{\lambda}^*$ for some $\boldsymbol{\lambda}^* \in I$ (note that this $\boldsymbol{\lambda}^*$ exists because $I$ is a compact metric space). In order to conclude, we have to prove that $\boldsymbol{\lambda}^* \in \Lambda(y_{\infty,1})$. 

    Denote with $y_{\infty, \Tilde{M}+1}(\boldsymbol{\lambda})$ the value of $y_{\infty,\Tilde{M}+1}$ obtained by starting from $y_{\infty,1}$ and following the strategy $\boldsymbol{\lambda}$. By definition, $y_{\infty, \Tilde{M}+1}(\boldsymbol{\lambda}_k) \le y^\textnormal{stop}$ for any $k$. Therefore, 
    \begin{equation}
        y_{\infty, \Tilde{M}+1}(\boldsymbol{\lambda}^*) = \lim_{k \to + \infty} y_{\infty, \Tilde{M}+1}(\boldsymbol{\lambda}_k) \le y^\textnormal{stop}, 
    \end{equation}
    hence $\boldsymbol{\lambda}^* \in \Lambda(y_{\infty,1})$. Consequently,  
    \begin{equation}
        T_{\boldsymbol{\lambda}^*}(y_{\infty,1}) \le \liminf_{k \to + \infty} T_{\boldsymbol{\lambda}_k}(y_{\infty,1}) = T(y_{\infty,1})
    \end{equation}
    because of the lower semicontinuity, while by definition of $T$ we have 
    \begin{equation}
        T(y_{\infty,1}) \le T_{\boldsymbol{\lambda}^*}(y_{\infty,1}). 
    \end{equation}
    This gives the desired equality and concludes the proof. 
\end{proof}

\begin{remark} \label{rmk:alternative_proof}
    An alternative proof of point \textit{(ii)} of Theorem \ref{teo:minimum_exists} directly follows by exploiting point (i) as follows. Consider the strategy $\Pi_\textnormal{fast} \circ \Pi_\textnormal{slow}^{\max}$, corresponding to a time $\Bar{T} < +\infty$. Then, $T(y_{\infty,1}) \le \Bar{T}$. Consider a sequence $\{\boldsymbol{\lambda}_k\}$ such that $T_{\boldsymbol{\lambda}_k}(y_{\infty,1}) \to T(y_{\infty,1})$ as $k \to +\infty$. Then, for sufficiently large $k$, 
    \begin{equation}
        \tau_{\min} M(\boldsymbol{\lambda}_k) \le T_{\boldsymbol{\lambda}_k}(y_{\infty,1}) \le \Bar{T}+1, 
    \end{equation}
    hence $M(\boldsymbol{\lambda}_k) \le \lfloor (\Bar{T}+1)/\tau_{\min}\rfloor$. Now it suffices to apply point (i) of Theorem \ref{teo:minimum_exists} with $M=\lfloor (\Bar{T}+1)/\tau_{\min}\rfloor$. 
\end{remark}

\begin{remark} \label{rmk:expl_strat}
    Consider the case $y^\textnormal{stop} \le y_\infty^\textnormal{cycle}$ and suppose that it is possible to reach $\hat z$ starting from $y_{\infty,1}$. Lemma \ref{lemma:stop} implies that $y_{\infty,1} \le - \sqrt{u_m/u_M}\Pi_\textnormal{fast}^{-1}(y^\textnormal{stop}) \le y^\textnormal{stop} \le y_\infty^\textnormal{cycle}$, hence we necessarily have $y_{\infty,2}>y_{\infty,1}$ and the target $\hat{z}$ must be reached during the first fast flow (between $y_{E,1}$ and $y_{\infty,2}$). Indeed, if this does not happen, the orbit would not be able to ``expand'' enough to reach $\hat{z}$ later. Therefore, $T(y_{\infty,1})=\tau(y_{\infty,1},\lambda_1)$ \eqref{eq:T_def}, where $\lambda_1$ minimizes the (continuous) function $\tau(y_{\infty,1},\cdot)$ over the (compact) interval of admissible values $[\lambda_1^*, \sqrt{u_M/u_m}]$, with $\lambda_1^*\in [\sqrt{u_m/u_M}, \sqrt{u_M/u_m}]$ obtained by ensuring that the exit point $y_{E,1}$ is equal or larger than $\Pi_{\textnormal{fast}}^{-1}(y^\textnormal{stop})$. 

    In this situation, it is possible to derive the explicit expression of the optimal control. Indeed, 
    \begin{align}
        \lambda_1^* = 
        \begin{cases}
            -\frac{1}{y_{\infty, 1}} \Pi_\textnormal{fast}^{-1} (y^\textnormal{stop}) \quad & \textnormal{if } -\frac{1}{y_{\infty, 1}} \Pi_\textnormal{fast}^{-1} (y^\textnormal{stop}) \ge \sqrt{\frac{u_m}{u_M}}, \\ 
            \sqrt{\frac{u_m}{u_M}} & \textnormal{otherwise}. 
        \end{cases}
    \end{align}
    Essentially, the first case corresponds to $y_{E,1} \ge \Pi_\textnormal{fast}^{-1} (y^\textnormal{stop})$, and hence $y_{\infty,2} \le y^\textnormal{stop}$, for $\lambda_1 \in [\lambda_1^*, \sqrt{u_M/u_m}]$. On the other hand, the second case corresponds to $y_{E,1} > \Pi_\textnormal{fast}^{-1} (y^\textnormal{stop})$ for any $\lambda_1 \in [\sqrt{u_m/u_M}, \sqrt{u_M/u_m}]$. Recall the consequences of Theorem \ref{teo:minimize}, and that $\lambda_1=1$ leads to $y_{E,1}=-y_{\infty,1}$, the optimal choice for $\lambda_1$ is 
    \begin{align}
        \lambda_1 = 
        \begin{cases}
            \lambda_1^* \quad & \textnormal{if } \lambda_1^* \ge 1, \\
            1 & \textnormal{otherwise.}
        \end{cases}
    \end{align}
    Finally, the corresponding values of $u$ and $T(y_{\infty,1})=\tau(y_{\infty,1}, \lambda_1)$ can be derived from Theorem \ref{teo:minimize}. 
\end{remark}

\subsubsection{(Non-)uniqueness of the optimal solution}

The optimal solution is never unique; indeed, the fast flow is independent on $u(t)$, hence it can attain any values during such flow, without affecting the evolution of the orbits. Anyway, a natural question is asking if the optimal strategy $\{\lambda_n\}$ is unique or not. As explicitly shown in Remark \ref{rmk:expl_strat}, when $y_{\infty,1} \le - \sqrt{u_m/u_M}\Pi_\textnormal{fast}^{-1}(y^\textnormal{stop}) \le y^\textnormal{stop} \le y_\infty^\textnormal{cycle}$ (i.e., when it is possible to reach the target $\hat{z}$ in the case $y^\textnormal{stop} \le y_\infty^\textnormal{cycle}$) the optimal strategy is unique. In the other case (i.e., $y^\textnormal{stop} > y_\infty^\textnormal{cycle}$), we are not able to prove (or disprove) analytically the uniqueness of the optimal strategy. We conjecture that the optimal strategy is not unique; however, a numerical search for a counterexample lies beyond the scope of this paper.

\subsubsection{Connection with the Bellman equation} 

Our optimal control problem can be formulated in a similar way to the Bellman equation \cite{bellman1954theory}. Indeed, for any admissible starting point $y_{\infty,1}$, Theorem \ref{teo:minimum_exists} implies that \eqref{eq:T_def} can be written as 
\begin{equation} \label{eq:T_classic}
    T(y_{\infty,1}) = \min_{\lambda_1} \{\tau(y_{\infty,1}, \lambda_1) + V(\Pi_\textnormal{fast}(-\lambda_1 y_{\infty, 1}))\}, 
\end{equation}
with 
\begin{align} \label{eq:W}
\begin{cases}
    V(y)=0 \quad &\textnormal{if } y \in (-b, y^\textnormal{stop}], \\
    V(y) = \displaystyle\min_\lambda \{\tau(y, \lambda) + V(\Pi_\textnormal{fast}(-\lambda y))\} \quad &\textnormal{if } y \in (y^\textnormal{stop}, 0).
\end{cases}
\end{align}
Such function $V$ had to be introduced because at least one entry-exit process must occur. Another important difference with the classical setting is that in such setting $\beta_n = \beta^n$ for some $\beta \in (0,1)$, making the infinite sum \eqref{eq:T_def} convergent. In the classical setting, $\lambda_1$ and $\lambda$ must belong to a set of admissible values depending only on $y_{\infty,1}$ and $y$, respectively. In our setting, their choices must also ensure to reach $\hat{z}$ (recall Remark \ref{rmk:infinite_sum}). Finally, we remark that our problem is perfectly in line with Bellman's \emph{Principle of Optimality} \cite{bellman_dynamic}: ``An optimal policy has the property that whatever the initial state and initial decision are, the remaining decisions must constitute an optimal policy with regard to the state resulting from the first decision''. 

\subsection{Numerical simulations}

Assume that it is possible to reach $\hat{z}$ starting from $y_{\infty,1}$. To numerically derive the optimal strategy we rely on the formulation \eqref{eq:T_classic}--\eqref{eq:W}. Define 
\begin{align}
    V_0(y) = 
    \begin{cases}
        0 \quad & \textnormal{if } y \in (-b, y^\textnormal{stop}], \\
        + \infty & \textnormal{if } y \in (y^\textnormal{stop},0), 
    \end{cases}
\end{align}
and for any $n \ge 1$ 
\begin{align}
    V_n(y) = 
    \begin{cases}
        0 \quad & \textnormal{if } y \in (-b, y^\textnormal{stop}], \\
        \displaystyle\min_\lambda \{\tau(y, \lambda) + V_{n-1}(\Pi_\textnormal{fast}(-\lambda y))\} & \textnormal{if } y \in (y^\textnormal{stop},0).  
    \end{cases}
\end{align}
Therefore, $V_n(y)$ represents the minimal slow time among all admissible strategies reaching the target in at most $n$ steps and eventually without doing any entry-exit processes; indeed, the first, necessary entry-exit process will be taken into account \emph{a posteriori} through the Eq. \eqref{eq:T_classic}. In other words, for any $n \ge 0$, the approximation of the original function $T$ is 
\begin{equation}
    T_n(y_{\infty,1}) = \min_{\lambda_1} \{\tau(y_{\infty,1}, \lambda_1) + V_n(\Pi_\textnormal{fast}(-\lambda_1 y_{\infty,1}))\}, 
\end{equation}
representing the minimal slow time among all admissible strategies reaching the target in at most $n+1$ steps. From Lemma \ref{lemma:N_finite} it follows that $T_n(y_{\infty,1})=+\infty$ for any $n < N(y_{\infty,1}) - 1$, while $T_n(y_{\infty,1}) < +\infty$ for any $n \ge N(y_{\infty,1}) - 1$. In conclusion, we directly get the following result, which ensures the convergence of this process. 

\begin{proposition}
    Assume that it is possible to reach $\hat{z}$ starting from $y_{\infty,1}$ and define $M(y_{\infty,1}) < + \infty$ as the minimum number of entry-exit phenomena associated to all possible optimal control strategies, then  
    \begin{equation}
        V_n(y_{\infty,1}) = V(y_{\infty,1}) \quad and \quad T_n(y_{\infty,1}) = T(y_{\infty,1})
    \end{equation}
    for all $n \ge M(y_{\infty,1}) - 1 \ge 0$. 
\end{proposition} 

Recall that when $y_{\infty,1} \le - \sqrt{u_m/u_M}\Pi_\textnormal{fast}^{-1}(y^\textnormal{stop}) \le y^\textnormal{stop} \le y_\infty^\textnormal{cycle}$ (i.e., when it is possible to reach $\hat{z}$ in the case $y^\textnormal{stop} \le y_\infty^\textnormal{cycle}$) we have $M(y_{\infty,1})=1$, hence the convergence of the iterative process is immediate. In the general case, Remark \ref{rmk:alternative_proof} and Eq. \eqref{eq:tau_min} provide an upper bound for $M(y_{\infty,1})$. 

\subsubsection{Examples} \label{sec:example}

We conclude this analysis with two numerical examples. We focus on the case $y^\textnormal{stop} > y_\infty^\textnormal{cycle}$ since the other has been fully characterized analytically. For both examples, we carried out simulations for $n=0, \dots, 4$.  

Figures \ref{fig:examples_2D_1} and \ref{fig:examples_2D_2} regard the first example: the former shows the values attained by $V_n$, while the latter the values of $T_n$. As we expected, $V_{n+1} \le V_n$ and $T_{n+1} \le T_n$. Note that, as $n$ increases, the larger is the largest value of $y_{\infty,1} \in (-b,0)$ starting from it is possible to reach $\hat{z}$. Moreover, there exist some $y_{\infty, 1}$ such that $T_3 < T_2$ and $T_4 < T_3$, implying that increasing the number of entry-exit phenomena could reduce the optimal time to reach the target $\hat{z}$. For example, set $y_{\infty,1} \approx -1.86$, the optimal strategy leads to $\hat{z}=9$ in $4$ steps: 
\begin{align}
\begin{aligned}
    y_{\infty.1} \approx -1.86 & \overset{\lambda_1 \approx 1.94}{\underset{\tau \approx 0.83}{\longrightarrow}} y_{E,1} \approx 3.61 \overset{\Pi_\textnormal{fast}}{\underset{\tau = 0}{\longrightarrow}} y_{\infty,2} \approx -2.77 \\ 
    & \overset{\lambda_2 \approx 2.52}{\underset{\tau \approx 1.88}{\longrightarrow}} y_{E,2} \approx 6.98 \overset{\Pi_\textnormal{fast}}{\underset{\tau = 0}{\longrightarrow}} y_{\infty,3} \approx -4.38 \\ 
    & \overset{\lambda_3 \approx 2.89}{\underset{\tau \approx 3.95}{\longrightarrow}} y_{E,3} \approx 12.7 \overset{\Pi_\textnormal{fast}}{\underset{\tau = 0}{\longrightarrow}} y_{\infty,4} \approx -5.99 \\ 
    & \overset{\lambda_4 \approx 3.11}{\underset{\tau \approx 6.81}{\longrightarrow}} y_{E,4} \approx 18.6 \overset{\Pi_\textnormal{fast}}{\underset{\tau = 0}{\longrightarrow}} y_{\infty,5} \approx -6.90 \approx y^\textnormal{stop} \\ 
    \implies  T_3(-1.86) &\approx 13.5. 
\end{aligned}
\end{align}
On the other hand, the best strategy that leads to the target in $3$ steps is not optimal: 
\begin{align}
\begin{aligned}
    y_{\infty.1} \approx -1.86 & \overset{\lambda_1 \approx 3.14}{\underset{\tau \approx 2.22}{\longrightarrow}} y_{E,1} \approx 6.28 \overset{\Pi_\textnormal{fast}}{\underset{\tau = 0}{\longrightarrow}} y_{\infty,2} \approx -3.91 \\ 
    & \overset{\lambda_2 \approx 3.13}{\underset{\tau \approx 4.64}{\longrightarrow}} y_{E,2} \approx 12.2 \overset{\Pi_\textnormal{fast}}{\underset{\tau = 0}{\longrightarrow}} y_{\infty,3} \approx -5.91 \\ 
    & \overset{\lambda_3 \approx 3.15}{\underset{\tau \approx 7.29}{\longrightarrow}} y_{E,3} \approx 18.6 \overset{\Pi_\textnormal{fast}}{\underset{\tau = 0}{\longrightarrow}} y_{\infty,4} \approx -6.90 \approx y^\textnormal{stop} \\ 
    \implies  T_2(-1.86) &\approx 14.2. 
\end{aligned}
\end{align}
However, for other values of $y_{\infty,1}$ we see that $T_{n+1}=T_n$ for all $n$. Notice also that the sequences of entry points is decreasing, as implied by Lemma \ref{lemma:tau_min}. 

Figures \ref{fig:examples_2D_3} and \ref{fig:examples_2D_4} show the second example. Interestingly, we see that $T(y_{\infty,1})$ does not have a monotone behavior. In particular, its minimum is reached for the value $y_{\infty,1}$ such that $\Pi_\textnormal{fast}(-y_{\infty,1})=y^\textnormal{stop}$; recall that fastest exit point to reach is the one opposite to the entry point. In the former example, there existed no $y_{\infty,1}$ satisfying such equality, leading to the increasing behavior of $T(y_{\infty,1})$.  

\begin{figure}[h!]
\centering
\begin{subfigure}{.45\textwidth}
  \centering
\begin{tikzpicture}
 \node at (0,0) {\includegraphics[width=.85\linewidth]{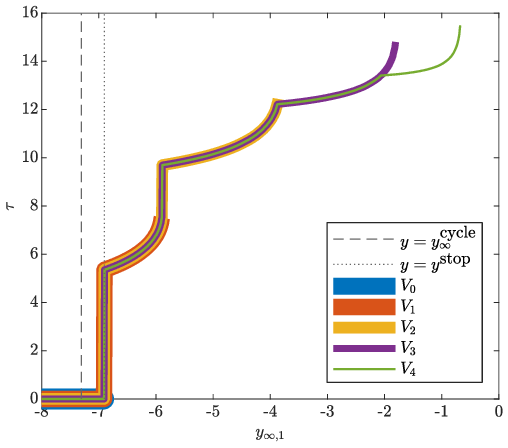}};
  \end{tikzpicture}
  \caption{}
  \label{fig:examples_2D_1}
\end{subfigure}\hspace{0.5cm}
\begin{subfigure}{.45\textwidth}
    \centering
\begin{tikzpicture}
 \node at (0,0) {\includegraphics[width=.85\linewidth]{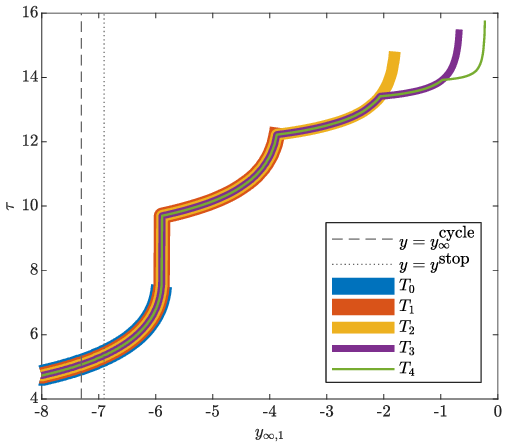}};
  \end{tikzpicture}
  \caption{}
  \label{fig:examples_2D_2}
\end{subfigure} \\ 
\vspace{.5cm} 
\begin{subfigure}{.45\textwidth}
   \centering
\begin{tikzpicture}
 \node at (0,0) {\includegraphics[width=.85\linewidth]{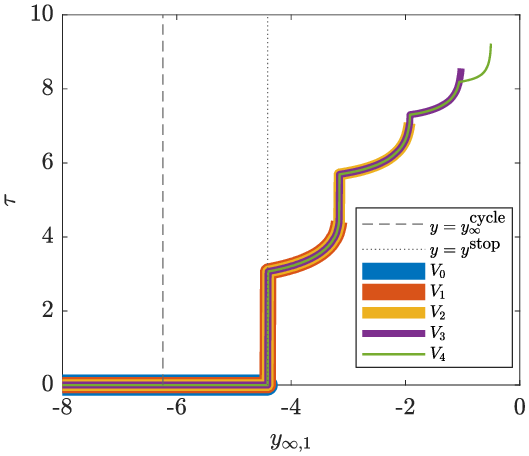}};
  \end{tikzpicture}
  \caption{}
  \label{fig:examples_2D_3}
\end{subfigure} \hspace{0.5cm}
\begin{subfigure}{.45\textwidth}
   \centering
\begin{tikzpicture}
 \node at (0,0) {\includegraphics[width=.85\linewidth]{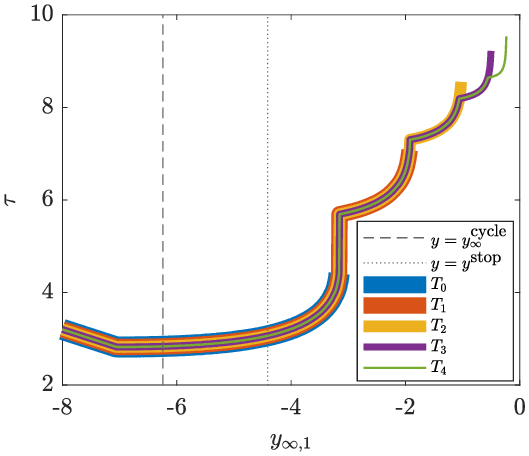}};
  \end{tikzpicture}
  \caption{}
  \label{fig:examples_2D_4}
\end{subfigure}
\caption{Graphs of $V_n$ (left) and corresponding graphs of $T_n$ (right) as $y_{\infty,1}$ varies between $-b$ and $0$, for $n=0, \dots, 4$. When a curve ``stops'', it means that it is not possible to reach the target $\hat{z}$ in the corresponding finite number of entry exit processes. (a), (b) Values of the parameters: $b=8$, $u_m=1$, $u_M=10$, and $\hat{z}=9$. (c), (d) Values of the parameters: $b=8$, $u_m=1$, $u_M=5$, and $\hat{z}=2$. In both cases $y^\textnormal{stop} > y_\infty^\textnormal{cycle}$, implying that the target could be reached starting from any $y_{\infty,1}$ in a sufficiently large number of entry-exit processes. 
\label{fig:examples_2D}}\end{figure}

\section{Extension to more general planar fast-slow models} \label{sec:other_planar}

This last section is devoted to extending the results that we obtained for the toy model \eqref{eq:model} to more complex planar fast-slow models. This provides a sketch of the general proof of Theorem \ref{teo:general} from the Introduction. In particular, some results reported in the mentioned theorem will be proved in an even more general setting, as in the first two steps we will consider a larger family of planar systems of ODEs, only restricting precisely to \eqref{eq:model_generic_teo1} when dealing with the optimal control of the entry-exit mechanism. 

We divide this analysis into five parts: in the first part, we describe the general structure of the models under consideration; in the following parts, we generalize each technique developed in the previous sections. In the following, we use the same notation exploited in the previous section, even if we are dealing with different objects, to highlight the parallelism between the two analyses. 

\bigskip 

\noindent \textbf{1) Model structure.} Consider the following planar model, evolving inside some $\Delta \subset \mathbb{R} \times [0, +\infty)$: 
\begin{align} \label{eq:model_generic}
\begin{aligned}
    \dot{y} &= \varepsilon g(y,z,\varepsilon,u(t)) + z h(y,z,\varepsilon), \\ 
    \dot{z} &= z f(y,z,\varepsilon), 
\end{aligned}
\end{align}
where $g$, $h$ and $f$ are differentiable. Let us now describe further assumptions on the model, and the motivations behind each of them. Recall that such assumptions are collected in Theorem \ref{teo:general}. 

We assume that the critical manifold coincides with the $y$-axis, meaning $\mathcal{C}_0=\{z=0\}$. Moreover, we assume that $\mathcal{C}_0$ is attractive for $y<0$, while it is repelling for $y>0$; in other terms, $\textnormal{sign}(f(y,0,0))=\textnormal{sign}(y)$. Finally, we assume that the fast subsystem of \eqref{eq:model_generic} (obtained by setting $\varepsilon=0$) generates the map 
\begin{equation} 
\Pi_{\textnormal{fast}} \colon \{y \in (0, +\infty) \cap \pi_y(\Delta)\} \to \{ y \in (-\infty, 0) \cap \pi_y(\Delta)\}, 
\end{equation}
i.e., the orbits of the fast flow are heteroclinic to two points (of opposite sign) on $\pi_y(\Delta)$, which is the projection on the $y$-axis of the forward invariant set $\Delta$ for system \eqref{eq:model_generic}. Importantly, $\Pi_{\textnormal{fast}}$ is necessarily a strictly decreasing function. Note that the fast flow is independent on $u(t)$, hence we can again focus on $U(\tau)=u(t)$. 

Concerning the slow dynamics, $g(y,0,0,U(\tau)) > 0$ explicitly depends on $U(\tau)$ (note that this implies that $O(1)$-variations of $U$ lead to $O(1)$-variations of $g(y,0,0,U(\tau))$). 

Finally, note that the arguments reported in Sections \ref{sec:unified_auton} and \ref{sec:unified} can be adapted to demonstrate a transition from the fast dynamics to the slow dynamics of system \eqref{eq:model_generic}. 

\bigskip 

\noindent \textbf{2) Non-autonomous entry-exit function.} Assume that an orbit reached a $\mathcal{O}(\varepsilon)$-neighborhood of $\mathcal{C}_0$ for some $y_\infty<0$ and rescale $z=\varepsilon x$. Moreover, change the time variable from $t$ to $\tau=\epsilon t$, and indicate $\text{d}/\text{d}\tau=:'$. Then, system \eqref{eq:model_generic} becomes
\begin{align} \label{eq:model_generic_slow}
\begin{aligned}
    y' &= g(y,\varepsilon x,\varepsilon,U(\tau)) + x h(y,\varepsilon x,\varepsilon), \\ 
    \varepsilon x'&= x f(y,\varepsilon x,\varepsilon).  
\end{aligned}
\end{align}
Evaluated on $\mathcal{C}_0$, this reduces to 
\begin{equation} 
    y' = g(y,0,\varepsilon,U(\tau)) = g(y,0,0,U(\tau)) + \mathcal{O}(\varepsilon),  
\end{equation}
while, more generally, as long as $x$ is $\mathcal{O}(\varepsilon)$-small, 
\begin{equation} 
    y' = g(y,\varepsilon x,\varepsilon,U(\tau)) + x h(y,\varepsilon x,\varepsilon) = g(y,0,0,U(\tau)) + \mathcal{O}(\varepsilon). 
\end{equation}
Consider an entry-exit process on the line $\{x =\Tilde{x}\}$, with $\Tilde{x}\in \mathcal{O}(\varepsilon)$; for any $\tau \in (0, \tau_E)$ we have 
\begin{equation}
    \varepsilon \log \frac{x(\tau)}{\Tilde{x}} = \int_0^\tau f(y(s), \varepsilon x(s), \varepsilon) \dd s. 
\end{equation}
We would like to remove the dependency on $x$ and $\varepsilon$ in the argument of this integral. We can write $f(y, \varepsilon x, \varepsilon)=f(y,0,0) + \mathcal{O}(\varepsilon)$. However, when $y$ is $\varepsilon$-close to $0$, the term $O(\varepsilon)$ cannot be neglected because $f(0,0,0)=0$. Nevertheless, the function $g$ implies that $y$ grows at a $\mathcal{O}(1)$-speed (on the slow timescale $\tau$), hence $y$ remains $\varepsilon$-close to $0$ only for $\varepsilon$-small times $\tau$. This in turn implies that we can write 
\begin{equation}
    \varepsilon \log \frac{x(\tau)}{\Tilde{x}} = \int_0^\tau f(y(s), 0, 0) \dd s + \mathcal{O}(\varepsilon). 
\end{equation}
Consequently, an $\mathcal{O}(\varepsilon)$-approximation of the (slow) exit time $\tau_E$ is given by the non-trivial solution of the integral equation 
\begin{equation} \label{eq:exit_time_generic}
    \int_0^{\tau_E} f(y(s), 0, 0) \dd s = 0, 
\end{equation}
where we stress that $y(s)$ depends on the choice of $U(\tau)$. Similarly, the exit point $y_E$ is given by 
\begin{equation} \label{eq:exit_point_generic}
    \int_{y_\infty}^{y_E} \frac{f(y, 0, 0)}{g(y,0,0,U(y))} \dd y = 0, 
\end{equation}
where we expressed $U$ in function of $y$ because the latter is strictly increasing during this process and hence invertible. 

\begin{remark}
    This analysis provides a non-autonomous counterpart to the classical autonomous entry-exit function. The key property we exploited is that $g(y,0,0,U(\tau))>0$ for any $U(\tau)$, which allowed us to deal with the non-hyperbolic point on the critical manifold. The expression we obtained is analogous to the classical autonomous case \cite{35}. Several studies have been devoted to characterizing the precise behavior of the orbit during the entry-exit process; in particular, we highlight the purely geometric proof of Hsu \cite{hsu2017bifurcation} (see the references therein for other proofs). The complexity of these analyses stems from the need to characterize the orbit precisely throughout the entry-exit process. In our case, we can avoid this level of detail because such information is not required for our purposes.
\end{remark}

\bigskip 

\noindent \textbf{3) Optimal control of each entry-exit process.} Given an entry point $y_\infty$, Eq. \eqref{eq:exit_point_generic} implies that the maximum (resp. minimum) exit point $y_E$ is obtained by choosing $U \in [u_m, u_M]$ such that it minimizes (resp. maximizes) $g(y,0,0,U)$ when $y<0$, while it maximizes (resp. minimizes) $g(y,0,0,U)$ when $y>0$. Note that such minima/maxima exist because $[u_m, u_M]$ is compact and at each time step $y$ is fixed; moreover, the differentiability of $g$ implies that a measurable $U$ can be chosen. Therefore, we can once again define the maps
\begin{equation} \label{eq:map_Pi_slow_max_generic}
    \Pi_{\textnormal{slow}}^{\max} \colon \{y \in (-\infty, 0) \cap \pi_y(\Delta)\} \to \left\{ y \in \left(0, +\infty\right) \cap \pi_y (\Delta) \right\}
\end{equation}
which maps $y_\infty$ to the maximum possible exit point $y_E^{\max}$, and 
\begin{equation} \label{eq:map_Pi_slow_min_generic}
    \Pi_{\textnormal{slow}}^{\min} \colon \{y \in (-\infty, 0) \cap \pi_y(\Delta)\} \to \left\{ y \in \left(0, +\infty\right) \cap \pi_y (\Delta) \right\}
\end{equation}
which maps $y_\infty$ to the minimum possible exit point $y_E^{\min}$. 

We can now demonstrate that, given an exit point $y_E \in [y_E^{\min}, y_E^{\max}]$, there exists an optimal control strategy leading to it in the minimum possible exit time $\tau_E$. To do so, we make an additional assumption on the function $g$. Namely, we assume that $g(y,z,\varepsilon,U(\tau))=U(\tau)\Tilde{g}(y,z,\varepsilon)$ (in principle, one could make weaker assumptions on $g$, however they would make the following analysis cumbersome without being particularly more interesting nor deeper from a mathematical point of view). In this scenario, the exit time is given by
\begin{equation}
    \tau_E = \int_{y_\infty}^{y_E} \frac{1}{U(y) \Tilde{g}(y,0,0)} \dd y, 
\end{equation}
and it must be minimized under the constraint 
\begin{equation}
    \int_{y_\infty}^{y_E} \frac{f(y,0,0)}{U(y) \Tilde{g}(y,0,0)} \dd y = 0.
\end{equation}
Now, fixed $\mu \in \mathbb{R}$ like in the proof of Theorem \ref{teo:minimize}, the quantity to be analyzed is $1 + \mu f(y,0,0)$. Therefore, following the proof of Theorem \ref{teo:minimize}, we immediately obtain that such optimal control exists and, for some $\Tilde{\mu}\in\mathbb{R}$, is uniquely determined almost everywhere outside the set $\{y \; \colon \; 1 + \Tilde{\mu} f(y,0,0) = 0\}$, 
where it attains values in $\{u_m,u_M\}$. In contrast to the toy model, this set may have positive measure. On this set, the pointwise optimality condition does not determine the value of $U$, potentially leading to non-uniqueness of the optimal control, provided that the entry-exit constraint is still satisfied. Nevertheless, since $f$ is continuous and $f(0,0,0)=0$, the complement of this set always has positive measure. Moreover, in this general setting, the equation $1+\Tilde{\mu}f(y,0,0)=0$ may admit multiple solutions, and hence the optimal control may have more than one switching point.

As it was the case in our toy model, the smaller is the entry point, the larger is the largest exit point that can be reached. This simply follows from the fact that smaller entry points allow the orbit to ``accumulate more contraction'' during the flow close to the attractive part of the critical manifold. However, the expression of all possible exit points is more complex than those related to the toy model. Recall that we were able to write $y_E=-\lambda y_\infty$ with $\lambda \in [\sqrt{u_m/u_M}, \sqrt{u_M/u_m}]$. In this general case, $\lambda \in [\lambda_{\min}(y_\infty), \lambda_{\max}(y_\infty)]$, namely such interval depends on the entry point. 

Finally, the continuity of $\lambda_{\min}$ and $\lambda_{\max}$ with respect to $y_\infty$ follows from the continuous dependence of the solutions of the entry-exit equation \eqref{eq:exit_point_generic} on the entry point. Analogously, the previous characterization of the minimum-time control, together with the continuous dependence of the corresponding integrals on the entry and exit points, implies that the optimal exit time $\tau(y_\infty, \lambda)$ is continuous with respect to its arguments.

\bigskip 

\noindent \textbf{4) Admissibility of the target.} Concerning the admissibility of a target $0 < \hat{z} \in \mathcal{O}(1)$, one simply has to analyze the orbits originating from the maximization of the exit points, i.e., $\Pi_{\textnormal{fast}} \circ \Pi_{\textnormal{slow}}^{\max}$. This directly follows from the monotonicity of $\Pi_{\textnormal{fast}}$. Importantly, it is not relevant whether these orbits converge toward a precise object (i.e., a limit cycle or an equilibrium) or not. 

Consider a starting entry point $y_{\infty,1}$ such that it allows to reach $\hat z$. As it was the case in our toy model, the strategy $\Pi_{\textnormal{fast}} \circ \Pi_{\textnormal{slow}}^{\max}$ allows to reach the target in the minimum possible number $N(y_{\infty,1})< + \infty$ of entry-exit phenomena. 

\bigskip 

\noindent \textbf{5) Existence of the optimal strategy.} The final optimal control problem can be formulated in a way similar to what we have done for our toy model. The only important difference lies in the spaces where all possible strategies $\boldsymbol{\lambda}=\{\lambda_n\}$ belong to. Recall that, for the toy model, such space was $I_M=[\sqrt{u_m/u_M}, \sqrt{u_M/u_m}]^M$ when considering at most $M$ entry-exit processes, while it was $I=[\sqrt{u_m/u_M}, \sqrt{u_M/u_m}]^\mathbb{N}$ when considering an arbitrary number of them; both of them are compact by Tychonoff's Theorem. However, as observed above, for a general model its expression is more complex. To expand Theorem \ref{teo:minimum_exists} to this general setting, we have to demonstrate that such spaces are still compact. The rest of the proof does not change. Indeed, thanks to our previous analyses, the function $\tau(\cdot, \cdot)$ is still well-defined and the formulation \eqref{eq:T_def} still holds.

Let $M \ge 1$ and define 
\begin{equation}
    I_M(y_{\infty,1}) \coloneqq \{\boldsymbol{\lambda} = (\lambda_1, \dots, \lambda_M) \in \mathbb{R}^M \colon \; \lambda_n \in [\lambda_{\min}(y_{\infty,n}), \lambda_{\max}(y_{\infty,n})] \; \textnormal{for all} \; n=1, \dots, M\}, 
\end{equation}
where we considered the continuous map (applied to $y_{\infty,1}$) 
\begin{equation} \label{eq:aux_bis}
    (\lambda_1, \dots, \lambda_n) \mapsto y_{\infty, n+1}. 
\end{equation}
We can prove that $I_M(y_{\infty,1})$ is compact by induction on $M$. The case $M=1$ is trivial. Assume that $I_M(y_{\infty,1})$ is compact and denote with $y_{\infty, n}(\boldsymbol{\lambda})$ the value of $y_{\infty,n}$ obtained by starting from $y_{\infty,1}$ and following the strategy $\boldsymbol{\lambda}$. The set of admissible values of $y_{\infty, M+1}$ is compact because \eqref{eq:aux_bis} is continuous. Therefore, on such set, the values of $\lambda_{\min}$ and $\lambda_{\max}$ are uniformly bounded by $0$ and some $\ell_{M+1}>0$ because they are continuous functions, in other words:
\begin{equation}
    0 < \min_{\boldsymbol{\lambda} \in I_M(y_{\infty,1})} \lambda_{\min}(y_{\infty, M+1}(\boldsymbol{\lambda})) < \max_{\boldsymbol{\lambda} \in I_M(y_{\infty,1})} \lambda_{\max}(y_{\infty, M+1}(\boldsymbol{\lambda})) \le \ell_{M+1}. 
\end{equation}
Now,  
\begin{equation}
    I_{M+1}(y_{\infty,1}) = \{(\boldsymbol{\lambda}, \lambda_{M+1}) \in I_M(y_{\infty,1}) \times [0, \ell_{M+1}] \colon \; \lambda_{\min}(y_{\infty,M+1}(\boldsymbol{\lambda})) \le \lambda_{M+1} \le \lambda_{\max}(y_{\infty,M+1}(\boldsymbol{\lambda}))\}, 
\end{equation}
which is closed because the functions 
\begin{align}
\begin{aligned}
    \psi^-_{M+1}(\boldsymbol{\lambda}) & \coloneqq \lambda_{M+1} - \lambda_{\min}(y_{\infty,M+1}(\boldsymbol{\lambda})), \\ 
    \psi^+_{M+1}(\boldsymbol{\lambda}) & \coloneqq \lambda_{\max}(y_{\infty,M+1}(\boldsymbol{\lambda})) - \lambda_{M+1},
\end{aligned}
\end{align}
are continuous. Therefore, $I_{M+1}(y_{\infty,1})$ is compact because it is a closed subset of the compact set $I_M(y_{\infty,1}) \times [0, \ell_{M+1}]$. 

Consider now the set  
\begin{equation}
    I(y_{\infty,1}) \coloneqq \{\boldsymbol{\lambda} \in \mathbb{R}^\mathbb{N} \colon \; \lambda_n \in [\lambda_{\min}(y_{\infty,n}), \lambda_{\max}(y_{\infty,n})] \; \textnormal{for all} \; n\}.  
\end{equation} 
Define the compact metric sets 
\begin{equation}
    L \coloneqq \prod_{n =1}^{+\infty} [0, \ell_n]  \quad \textnormal{and} \quad L_M \coloneqq \prod_{n =1}^M [0, \ell_n], 
\end{equation}
and denote with $\pi_M \colon L \to L_M$ the projection on the first $M$ coordinates. Then, 
\begin{equation}
    I(y_{\infty,1}) = \bigcap_{M=1}^{+\infty} \pi_M^{-1}(I_M(y_{\infty,1})). 
\end{equation}
Since each $\pi_M$ is continuous, then $I(y_{\infty,1})$ is a closed subset of $L$, and hence compact. 

This completes the proof of Theorem \ref{teo:general}. Importantly, note that an additional assumption regarding $\tau_{\min}$ (recall Lemma \ref{lemma:tau_min}) is necessary. Indeed, we are not able to prove such result in this general setting, since it depends on the particular structure of the model under consideration. However, this assumption is intuitively obvious, since, again, the exit time is equal to zero if and only if $y_\infty=0=y_E$ (recall Remark \ref{rmk:tau_min}). Clearly, it is not possible to derive explicit expressions for the optimal controls and the optimal strategy in this general setting. 

\begin{remark}
    The optimal controls and strategies we presented in this paper were derived in the (singular) limit $\varepsilon \to 0$. However, they provide $\mathcal{O}(\varepsilon |\log \varepsilon|)$-approximations of the strategies that would correspond to a fixed $\varepsilon \ll 1$. Indeed, the formulas for the exit time and the exit point lead to $\mathcal{O}(\varepsilon)$-errors, and the same occurs when looking at the fast subsystems. However, some $\mathcal{O}(\varepsilon |\log \varepsilon|)$-errors could derive from the transitions between fast and slow dynamics. Since the dynamics evolve for finite slow times, the combination of all these errors is of order $\varepsilon |\log \varepsilon|$ (on the slow timescale $\tau$). Clearly, in this case one must be satisfied with reaching a target which is $\mathcal{O}(\varepsilon |\log \varepsilon|)$-close to a given target $\hat{z}$. 
\end{remark}

\appendix 
\section{Infinite entry-exit phenomena in a finite time} \label{appA} 

Consider our toy model \eqref{eq:model}. In this appendix, we want to prove that there exist some strategies $\{\lambda_n\}$ not leading to the target $\hat{z}$ that however satisfy $\sum_{n=1}^{+\infty} \tau(y_{\infty,n}, \lambda_n) < +\infty$. In other words, an infinite number of entry-exit phenomena could last a finite (slow) time. 

Since $y_{\infty,n+1} = \Pi_\textnormal{fast}(-\lambda_n y_{\infty,n}) > \lambda_n y_{\infty, n}$ \eqref{eq:pi_fast_aux}, then 
\begin{equation}
    y_{\infty,n+1} > y_{\infty,1} \prod_{j=1}^n \lambda_j. 
\end{equation}
Therefore, recalling the consequences of Theorem \ref{teo:minimize}, 
\begin{equation}
    \sum_{n=1}^{+\infty} \tau(y_{\infty,n}, \lambda_n) = - \sum_{n=1}^{+\infty}y_{\infty,n} \phi(\lambda_n) \le - y_{\infty,1} \phi_{\max} \sum_{n=1}^{+\infty} \prod_{j=1}^{n-1} \lambda_j. 
\end{equation}
If, for example, $\lambda_n \le q < 1$, then 
\begin{equation}
    \sum_{n=1}^{+\infty} \tau(y_{\infty,n}, \lambda_n) \le -\frac{y_{\infty,1} \phi_{\max}}{1-q} < + \infty. 
\end{equation}
Note that, in this scenario, $y_{\infty,n} \to 0$ as $n \to +\infty$. 

\begin{remark}
    One might argue that, since the transition from the fast to the slow dynamics of \eqref{eq:model} occurs only when $y_{\infty,n}$ is not $\varepsilon$-close to $0$, entry points satisfying $y_{\infty,n}\to 0$ should not be considered. In other words, the $y_{\infty,n}$ have to be bounded away from $0$ by some quantity arbitrary small, because we are considering the limit $\varepsilon \to 0$. However, this does not pose a problem. Our strategy was to reformulate the original problem of reaching the target $\hat{z}$ under the flow of \eqref{eq:model} in terms of \eqref{eq:T_def}. The geometric analysis showed that the two formulations are equivalent as long as the $y_{\infty,n}$ are not $\varepsilon$-close to $0$. To avoid sequences such that $y_{\infty,n} \to 0$, we imposed the requirement that the sequence $\{\lambda_n\}$ leads to the target (recall Section \ref{sec:formulation}). The present analysis demonstrate that such requirement is necessary, otherwise solutions to \eqref{eq:T_def} could not be solution of the original problem related to \eqref{eq:model}. Moreover, Lemma \ref{lemma:tau_min} demonstrated that any strategy $\{\lambda_n\}$ leading to the target in a time arbitrary close to the infimum \eqref{eq:T_def} is necessarily bounded away from $0$. This demonstrates the well-posedness of our approach. 
\end{remark}

\bigskip
\bigskip
\noindent\textit{Acknowledgments.} ChatGPT (OpenAI, v. 5.6) was used to double-check and improve Sections \ref{sec:opt_cont} and \ref{sec:other_planar}. The authors take full responsibility for the content of the manuscript. JB and MS are members of and acknowledge the support of the {\it Gruppo Nazionale di Fisica Matematica} (GNFM) of the {\it Istituto Nazionale di Alta Matematica} (INdAM).

{\footnotesize
	\bibliographystyle{unsrt}
	\bibliography{biblio}
}

\end{document}